\documentclass[hidelinks,onefignum,onetabnum]{siamart220329}

\makeatletter
\def\BState{\State\hskip-\ALG@thistlm}
\makeatother
 
\usepackage{color}
\usepackage{subfigure}
\usepackage{enumitem}

\usepackage{stmaryrd}
\setlist[itemize]{leftmargin=*}
\usepackage{amsmath}
\usepackage{mathtools}
\usepackage{amssymb}
\usepackage{commath}
\usepackage{bigints}
\usepackage{romanbar}
\usepackage{algorithm}

\usepackage{pgfplotstable}
\usepackage{pgfplots}
  \pgfplotsset{compat=newest}
  \usetikzlibrary{plotmarks}
  \usetikzlibrary{arrows.meta}
  \usepgfplotslibrary{patchplots}
  \usepackage{grffile}

\newtheorem{thm}{Theorem}[section]

\newtheorem{cor}{Corollary}[section]
\newcommand{\vx}{\mathbf{x}}

\newcommand{\vy}{\mathbf{y}}

\newcommand{\vv}{\mathbf{v}}
\newcommand{\vw}{\mathbf{w}}
\newcommand{\vf}{\mathbf{f}}

\newcommand{\vvarphi}{\boldsymbol{\varphi}}

\newcommand{\vnu}{\boldsymbol{\nu}}

\newcommand{\avrg}[1]{{\left\{\kern-0.5ex\left\{ #1 \right\}\kern-0.5ex\right\}}} % average
\newcommand{\Th}{\mathcal{T}_h} % mesh
\newcommand{\Eh}{\mathcal{E}_h} % set of edges
\newcommand{\Nh}{\mathcal{N}_h} % set of vertices
\newcommand{\V}{\mathbb{V}} 
\newcommand{\tol}{ {\rm tol}} 
\newcommand{\bz}{\boldsymbol{0}}
\newcommand{\g}{g} % first fundamental form
\newcommand{\tr}{{\rm tr}} % trace

\newcommand{\A}{\mathcal{A}} % mesh

\newcommand{\I}{\normalfont{\Romanbar{1}}}

\usepackage{lipsum}
\usepackage{amsfonts}
\usepackage{graphicx}
\usepackage{epstopdf}
\usepackage{algorithmic}
\ifpdf
  \DeclareGraphicsExtensions{.eps,.pdf,.png,.jpg}
\else
  \DeclareGraphicsExtensions{.eps}
\fi

\newsiamremark{remark}{Remark}
\newsiamremark{hypothesis}{Hypothesis}
\crefname{hypothesis}{Hypothesis}{Hypotheses}
\newsiamthm{claim}{Claim}

\headers{Accelerated flows for nonlinear plates}{G. Dong, H. Guo, and S. Yang}

\title{Accelerated gradient flows for large bending deformations of nonlinear plates}%\thanks{
\author{Guozhi Dong\thanks{School of Mathematics and Statistics, Hunan Research Center of the Basic
Discipline for Analytical Mathematics, HNP-LAMA, Central South University, Changsha 410083, China 
  (\email{guozhi.dong@csu.edu.cn}).}
\and Hailong Guo\thanks{School of Mathematics and Statistics, The University of Melbourne, Parkville, VIC, 3010, Australia 
  (\email{hailong.guo@unimelb.edu.au}).}
\and Shuo Yang\thanks{Beijing Institute of Mathematical Sciences and Applications, Beijing, 101408, China
  (\email{shuoyang@bimsa.cn}).}}

\usepackage{amsopn}

\ifpdf
\hypersetup{
  pdftitle={Accelerated flows for nonlinear plates},
  pdfauthor={G. Dong, H. Guo, and S. Yang}
}
\fi

\begin{document}
\maketitle

% REQUIRED
\begin{abstract}
In this paper, we propose novel algorithms integrated projection-free techniques with accelerated gradient flows to minimize bending energies for nonlinear plates with non-convex metric constraints. We discuss the stability and constraint consistency in a semi-discrete setting for both bilayer and prestrained plates.  The proposed algorithms exhibit substantial improvements in efficiency and accuracy compared to the current state-of-the-art methods based on gradient flows.
\end{abstract}

% REQUIRED

% REQUIRED

\section{Introduction}
Large spontaneous deformations of nonlinear plates have attracted significant attention in materials engineering and applied mathematics. These slender structures respond to environmental stimuli, such as temperature or light, enabling diverse applications across natural phenomena and engineered devices at various scales. Examples include the snapping behavior of the Venus flytrap \cite{forterre2005venus}, the natural growth of plant leaves and blossoms \cite{goriely2005differential}, cell encapsulation devices \cite{stoychev2012shape}, and self-deploying solar sails \cite{love2007demonstration}.

Bilayer and prestrained plates are models capable of describing deformation mechanisms for a broad class of structures. Bilayer plates consist of two materials glued together, which react differently to environmental changes, allowing them to attain non-trivial shapes without external forces. Prestrained plates, on the other hand, simulate materials with residual stress (natural or manufactured) in their ``rest'' state, leading to large spontaneous deformations. Both models are derived from 3D nonlinear hyperelasticity theory and involve minimizing energies under non-convex metric constraints. For a comprehensive discussion of these models and their applications, we refer to \cite{bartels2017bilayer}, \cite{bonito2022ldg}, and references therein.

This work focuses on developing novel computational methods for these challenging non-convex constrained variational problems. Throughout the paper, material deformations are denoted by $\vy: \Omega \to \mathbb{R}^3$, where $\Omega \subset \mathbb{R}^2$ is a bounded Lipschitz domain. Equilibrium deformations are characterized as minimizers of bending energies subject to metric constraints on $\vy$.

\subsection{Bilayer plates}
The model of bilayer plates we consider here was developed and analyzed in \cite{schmidt2007minimal, schmidt2007plate,bartels2017bilayer}. 
In this model, the equilibrium deformations are characterized by minimizers of the following constrained optimization problem 
\begin{equation}\label{bilayer-energy-3}
\min\limits_{\mathbf{y}\in\A}E_{bi}\left[\mathbf{y}\right]:= \min\limits_{\mathbf{y}\in\A}\frac{1}{2}\int_{\Omega}\big|D^2\vy\big|^2-\sum_{i,j=1}^2\int_{\Omega}\partial_{ij}\vy\cdot(\partial_1\vy\times\partial_2\vy)Z_{ij},
\end{equation}
where $Z\in[L^{\infty}(\Omega)]^{2\times2}$ represents the \emph{spontaneous curvature} that reflects the materials characteristics of the bilayer plates. Essentially, the presence of $Z$ drives the plate $\vy(\Omega)$ to deform out of plane and achieve non-trivial shapes. 
We further define the non-convex term of $E_{bi}$ as $E_{bi}^{nc}:=-\sum_{i,j=1}^2\int_{\Omega}\partial_{ij}\vy\cdot(\partial_1\vy\times\partial_2\vy)Z_{ij}$, and convex terms as $E^c_{bi}=E_{bi}-E^{nc}_{bi}$.

The {\it admissible set} $\A$, which prevents shearing and stretching within the surface $\vy(\Omega)$ and imposes possible boundary conditions, is defined as follows: 
\begin{equation}\label{iso-admissible}
\A:=\big\{\mathbf{y}\in[H^2(\Omega)]^3:\quad\I[\vy]=I_2\ \text{ in }\Omega,\quad \mathbf{y}=\vvarphi, \ \nabla\mathbf{y}=\Phi\text{ on }\Gamma^D\big\}.
\end{equation}
Here $\I[\vy]$ is the first fundamental form of $\vy(\Omega)$, i.e. 
\begin{equation}\label{first-fund-form}
\I[\vy]:=\nabla\mathbf{y}^T\nabla\mathbf{y},
\end{equation}
and $I_2$ is the $2\times2$ identity matrix, $\Gamma^D\subset\partial\Omega$ stands for the part of boundary where clamped boundary conditions are imposed.
We assume that the boundary data $\vvarphi\in [H^2(\Omega)]^3$ and $\Phi \in [H^1(\Omega)]^{3\times2}$ are given and compatible with the isometry constraint, namely $\Phi=\nabla\vvarphi$ and $\Phi^T\Phi=I_2$ on $\Gamma_D$; thus, $\A$ is non-empty. It is worth noting that both the bending energy $E_{bi}[\vy]$ and the constraint $\I[\vy]=I_2$ are \emph{non-convex} in this model. 

\subsection{Prestrained plates}
For prestrained plates, we define a target metric $g\in[H^1(\Omega)\cap L^{\infty}(\Omega)]^{2\times2}$, a given symmetric positive definite matrix-valued function, and the deformations belong to the following admissible set
\begin{equation}\label{pre-admissible}
\A_g:=\big\{\mathbf{y}\in[H^2(\Omega)]^3:\quad\I[\vy]=g\ \text{ in }\Omega,\quad \mathbf{y}=\vvarphi, \ \nabla\mathbf{y}=\Phi\text{ on }\Gamma^D\big\}.
\end{equation}
Moreover, we assume that $g$ admits \emph{isometric immersions}, i.e., there exists $\vy\in[H^2(\Omega)]^3$ such that $\nabla\vy^T\nabla\vy=g$ a.e., and again the boundary data $\vvarphi\in [H^2(\Omega)]^3$ and $\Phi \in [H^1(\Omega)]^{3\times2}$ are compatible with $g$, namely $\Phi=\nabla\vvarphi$ and $\Phi^T\Phi=g$ on $\Gamma_D$. These assumptions guarantee that the admissible set $\A_g$ is nonempty.   

We consider the bending energy model for isotropic prestrained plates that was proposed in \cite{efrati2009elastic} and later rigorously analyzed in \cite{bhattacharya2016plates}. We refer to \cite{bonito2022ldg} for a formal derivation. 
The equilibrium deformation is a solution of 
\begin{equation}\label{prob:min_Eg}
\min\limits_{\mathbf{y}\in\A_g}E_{pre}\left[\mathbf{y}\right]:= \min\limits_{\mathbf{y}\in\A_g} \frac{\mu}{12}\int_{\Omega}\left|\g^{-\frac{1}{2}}D^2\vy\g^{-\frac{1}{2}}\right|^2+\frac{\lambda}{2\mu+\lambda}\tr\left(\g^{-\frac{1}{2}}D^2\vy \g^{-\frac{1}{2}}\right)^2,
\end{equation}
where $\lambda$ and $\mu$ are Lam\'e parameters of the material. 
For the sake of concise presentation, although $D^2\vy$ is a $3\times2\times2$ tensor valued, we here define $\g^{-\frac{1}{2}}D^2\vy\g^{-\frac{1}{2}}:=\big(\g^{-\frac{1}{2}}D^2y_m\g^{-\frac{1}{2}}\big)_{m=1}^3$ with $\vy := (y_m)_{m=1}^3$, and this is also a $3\times2\times2$ tensor.
Note that in \eqref{prob:min_Eg} the energy $E_{pre}$ is quadratic (thus convex), while the constraint is still \emph{non-convex}. 

\subsection{Numerical methods for nonlinear plates}\label{sec:prev-works}
Numerically solving the constrained problems \eqref{bilayer-energy-3} and \eqref{prob:min_Eg} presents a nontrivial challenge, primarily  due to their inherent non-convexity. 

From the aspect of optimization, most existing works for nonlinear plates, such as \cite{bartels2013approximation,bonito2019dg,bonito2023numerical}, resort to gradient flow type iterative schemes along with tangent space update strategy to linearize the constraint $\I[\vy]=g$ in each step.
The semi-discrete formulation formally reads as  
\begin{equation}\label{eq:gf-semi-dis}
\tau^{-1}\delta\vy^{n+1}+\delta E[\vy^{n}+\delta\vy^{n+1}]=\bz,
\end{equation}
where $\vy^n$ is a given previous step of iteration, $\tau>0$ is a time step parameter,  and $\delta E$ represents the first variation of energy functional. In these schemes, $\delta\vy^{n+1}:=\vy^{n+1}-\vy^n\in \mathcal{F}(\vy^n)$ is solved within tangent space of the constraint at previous iterate: 
\begin{equation}\label{def:tangent-space}
\mathcal{F}(\vy^n):=\big\{\vv:L[\vy^n;\vv] := \nabla \vv^T \nabla \vy^n + (\nabla \vy^n)^T \nabla \vv = \bz\big\}.
\end{equation}
This method naturally compares to the projection-free gradient flows in the study of harmonic maps and liquid crystals; see \cite{bartels2016projection,nochetto2022gamma}. 

For bilayer plates, where the energy is non-convex, the above gradient flow step \eqref{eq:gf-semi-dis} is nonlinear. In \cite{bartels2017bilayer}, a fixed point sub-iteration is employed to solve the nonlinear equation \eqref{eq:gf-semi-dis}. In contrast, \cite{bartels2020stable,bonito2023gamma} treat the non-convex part of the energy explicitly and keep the convex part implicit, which solves linear problems iteratively in the end.  

Although the constraint is linearized and thus violated in iterations, the violation is proven to be proportional to the time step $\tau$ in the aforementioned works. Moreover, monotone decay of the energy is guaranteed in these gradient flow methods, either unconditionally or with mild restrictions on $\tau$. 

From the aspects of spatial discretization, several finite element methods (FEM), including Kirchhoff FEM, interior penalty discontinuous Galerkin (IPDG) method, Specht FEM, and local discontinuous Galerkin (LDG),  have been designed for these nonlinear plates models and analyzed in the framework of $\Gamma$-convergence; see \cite{bartels2013approximation,bonito2019dg,li10specht,bartels2017bilayer,bartels2018modeling,bartels2020stable,bonito2020discontinuous,bonito2023gamma,bonito2022ldg,bonito2023numerical,bonito2023numerical-2}. 

Recently, a deep learning based method was designed in \cite{li2023pre} for bilayer plates \eqref{bilayer-energy-3}, and it utilizes a stochastic gradient descent method and deals with the constraint by adding a penalty term to the energy. Moreover, the algorithm in \cite{li10specht} employs the combination of an adaptive time-stepping gradient flow and a Newton’s method. Additionally, \cite{rumpf2022finite} incorporates the nonlinear metric constraint by Lagrange multipliers and solves the nonlinear problem via a Newton's method for a shell model similar to prestrained plates.   

\subsection{Accelerated gradient flows}
Although gradient flow schemes are popular in existing works due to their ease of use and computational stability, they converge slowly for difficult problems like bilayer plates, where the non-convexity brings a more significant challenge, and prestrained plates, when the anisotropy in the metric tensor becomes significant. On the other hand, gradient-flow-based algorithms are easily stuck at stationary points for non-convex problems. With an attempt to overcome these limitations, we aim to design efficient and robust algorithms based on accelerated gradient flows for \eqref{bilayer-energy-3} and \eqref{prob:min_Eg} in this work.    
       
 The study of accelerated gradient flows can be traced back to the seminal Polyak's momentum method \cite{Pol64,AttGouRed2000} and also Nesterov's accelerated gradient descent method \cite{Nesterov1983,Nesterov2004} for convex optimization problems. More recently, an ordinary differential equation interpretation of the Nesterov's method in finite dimensional space was eastablished in \cite{SuBoyCan2016JMLR}, which inspired more advanced theoretical analysis and generalization; see, for instance, the works in \cite{AttChbPeyRed2018,AttBotCse2023JEMS,BotDonElbSch2022FCM,CheDonIglLiuXie25,DonHinZha2021SIIMS} and the references therein. 

The established theory of accelerated flows in convex optimization cannot be applied directly to non-convex variational models. Research on these accelerated flows for variational problems in scientific computing with practical applications, which often involving non-convex models, is still in its early stages both theoretically and computationally. For example, Nesterov-type acceleration has been utilized in an unconstrained model of prestrained plates \cite{bonito2023numerical-2,bonito2023finite}, and similar techniques are explored for a broad class of unconstrained problems \cite{calder2019pde,CheDonIglLiuXie25}. In the context of constrained models, novel accelerated flow algorithms show promising efficiency in computing ground states for Bose-Einstein condensates \cite{CheDonLiuXie23} and the Landau-de Gennes model for liquid crystals \cite{AAMM-14-1} without numerical analysis. Notably, these methods deal with different constraints using distinct approaches compared with our work.  
  
In this work, we propose novel algorithms for constrained minimization problems \eqref{bilayer-energy-3} and \eqref{prob:min_Eg} by integrating the acceleration technique with the tangent space update strategy \eqref{def:tangent-space} that handles metric constraints. 
The accelerated flows entail modifying gradient flows \eqref{eq:gf-semi-dis} to be formally as 
\begin{equation}\label{eq:second-semi-dis}
\tau^{-2}(\delta\vy^{n+1}+\vw^n-\vy^n)+\delta E[\delta\vy^{n+1}+\vw^n]=\bz,
\end{equation}
where $\vw^{n}:=\vy^{n}+\eta^{n}\delta\vy^{n}$ is an auxiliary variable, and $\eta^n$ is a crucial coefficient that influences the acceleration effect and constraint violation. The choice $\eta^n=\frac{n-1}{n+\alpha-1}$ with $\alpha\ge3$ is motivated by a generalized Nesterov's acceleration strategy, and $\eta^n=1-\beta\tau$ with $0<\beta<1/\tau$ is inspired by the heavy ball method. 
Moreover, we emphasize that in \eqref{eq:second-semi-dis}, $\delta\vy^{n+1}\in\mathcal{F}(\vy^n)$ is computed with the linearized constraint $L[\vy^n;\delta\vy^{n+1}]=\bz$. In particular, for bilayer plates, we treat the convex part $E^{c}_{bi}$ of energy implicitly and the non-convex part $E^{nc}_{bi}$ explicitly.   

\subsection{Contributions and outline of the paper}\label{sec:outline}
In Section \ref{sec:acc-flow}, we introduce a  {semi-implicit} accelerated gradient flow scheme incorporating a tangent space update strategy for prestrained plates. We establish a formal relationship between the accelerated flows and a dissipative hyperbolic partial differential equation involving the second-order time derivative of $\vy$, as in \eqref{eq:ODE}. 
This relationship motivates our result for the monotonic decrease in total energy (the sum of $E_{pre}$ and a kinetic term). Unlike gradient flows, the energy $E_{pre}$ oscillates in accelerated flows, while the kinetic energy converges to $0$ as $n \to \infty$. Furthermore, we demonstrate that constraint violations are controlled by $\mathcal{O}(\tau)$.

Section \ref{sec:backtrack} extends our accelerated flows by incorporating a backtracking strategy to ensure the monotone decrease of $E_{pre}$. In Section \ref{sec:BDF2}, we introduce a second-order backward differential formula (BDF2) approximation to improve the precision of constraint preservation, inspired by recent work \cite{akrivis2023quadratic} that employed the BDF2 scheme in gradient flows with tangent space updates and appropriate extrapolation. 

 {In Section~\ref{sec:bilayer}, we propose implicit-explicit accelerated gradient flows for bilayer plates, employing an implicit-explicit splitting to handle the additional non-convex energy term.}
We reiterate the results on total energy stability and constraint violation control, of which detailed proofs are provided in Appendix \ref{sec:appendix}.
 {The backtracking strategy and BDF2 temporal discretization are also applicable to the implicit-explicit accelerated gradient flow method.}

 {Section~\ref{sec:alg} presents pseudo-code implementations of all proposed algorithms. Notably, }
the discussions in this paper primarily focus on temporal semi-discrete formulations but are applicable to general spatial discretizations with appropriate consistency. 
We specifically employ Morley FEM for spatial discretization in our computations, as introduced in Section \ref{sec:alg} as well. 

In Section \ref{sec:numerics}, we evaluate the proposed algorithms through several benchmark examples. Section \ref{sec:numerics-properties} confirms our theoretical findings, including total energy stability and control of constraint violations. In Section \ref{sec:varying-alpha}, we discuss the selection principles for key parameters $\alpha$ and $\beta$ in our algorithms. Section \ref{sec:comparison} illustrates that our proposed algorithms outperform existing gradient flow schemes, particularly for challenging problems such as bilayer plates or prestrained plates with a large ratio of eigenvalues in the metric $g$. We find that accelerated flows with backtracking and BDF2 schemes are the most efficient and accurate among all discussed algorithms. Finally, in Section \ref{sec:local-global}, we demonstrate the capability of the new method to reach a global minimizer for a challenging bilayer plate example, where previous gradient-flow methods were known to be stuck at a stationary point.

In Section \ref{sec:conclusions}, we summarize our contributions and outline potential future directions related to the current work. 

\section{ {Semi-implicit accelerated gradient flows}}
\label{sec:acc-flow}
 {
Motivated by \eqref{eq:second-semi-dis}, we develop semi-implicit semi-discrete formulations of accelerated flows for prestrained plates, specifically targeting the constrained minimization problem \eqref{prob:min_Eg} and analyzing its theoretical properties including energy stability and constraint violation control.  
}

We fix the notations $\|\cdot\|_{L^p}$ ($1\le p\le\infty$), $\|\cdot\|_{H^2}$, $|\cdot|_{H^2}$, $(\cdot,\cdot)_{L^2}$, and $(\cdot,\cdot)_{H^2}$ to represent norms, seminorms, and inner products on Sobolev spaces such as $L^p(\Omega)$ and $H^2(\Omega)$, as well as their vector-valued and tensor-valued counterparts. We will slightly abuse the inner product notation by defining $(\vv,\vw)_{H^2}:=(D^2\vv,D^2\vw)_{L^2}$ whenever $\vv,\vw\in[H^2_0(\Omega)]^3$, since $\|\vv\|_{H^2}\approx|\vv|_{H^2}$ holds for any $\vv\in[H^2_0(\Omega)]^3$. 
We use $C(\cdot)$ to represent generic constants, which are determined by the quantities in the parentheses. In particular, the use of the same $C(\cdot)$ indicates their dependence on the same data, but does not necessarily mean they are identical. Moreover, we define bilinear forms relates to first variations of prestrained energy as $a_g(\vy,\vv):=\delta E_{pre}[\vy](\vv)$. For any $\vv\in[H^2(\Omega)]^3$,  
\begin{equation}\label{eq:coer-ag}
a_g(\vv,\vv)=2E_{pre}[\vv]\ge C(\lambda,\mu,g)|\vv|_{H^2}^2.
\end{equation} 

We utilize the tangent space update strategy as described in Section \ref{sec:prev-works}, which results in a violation of constraint that we aim to control.    
Therefore, we define the violation $D_{g,p}$ with the given metric $g$ and $1\le p\le\infty$ and corresponding relaxed admissible sets $\A_{g,p}^\epsilon$ as  
\begin{equation}\label{eq:Dgp}
D_{g,p}[\vy]:=\|\I[\vy]-g\|_{L^p},
\end{equation}
and
\begin{equation}\label{pre-admissible-relaxed}
\A_{g,p}^\epsilon:=\big\{\mathbf{y}\in[H^2(\Omega)]^3:\quad D_{g,p}[\vy]\le\epsilon,\quad \mathbf{y}=\vvarphi, \ \nabla\mathbf{y}=\Phi\text{ on }\Gamma^D\big\}.
\end{equation}
It is clear that in the limit of $\epsilon\to0$,  the metric constraints $\I[\vy]=g$ are satisfied pointwise almost everywhere. 

\subsection{ {BDF1 iterative scheme}}\label{sec:prestrained} 
Choosing a proper initial state $\vy^0\in\A_{g}$ satisfying the metric constraint and Dirichlet boundary conditions as in \eqref{pre-admissible} and setting $\vw^0=\vy^0$, at arbitrary $n$-th iteration ($n\ge0$) we compute increments $\delta\vy^{n+1}\in\mathcal{F}(\vy^n)$,  {where BDF1 scheme is employed for approximating the first-order temporal derivative}, so that 
\begin{align}\label{eq:nesterov_scheme}
\tau^{-2}(\delta\vy^{n+1},\vv)_{H^2} + a_g(\delta\vy^{n+1},\vv)&=-a_g(\vw^n,\vv)+\tau^{-2}(\vw^n-\vy^{n},\vv)_{H^2} \\ \label{eq:update_y}
\vy^{n+1}&=\vy^{n}+\delta\vy^{n+1} \\ \label{eq:update_w}
\vw^{n+1}&=\vy^{n+1}+\eta^{n+1}\delta\vy^{n+1}.
\end{align}
 {The damping coefficient} $\eta^n$ has two choices: a varying coefficient  
\begin{equation}\label{eq:eta-nesterov}
\eta^n:=\frac{n-1}{n+\alpha-1},
\end{equation}
with $\alpha\ge3$, or a constant coefficient  
\begin{equation}\label{eq:eta-heavy-ball}
\eta^n\equiv1-\beta\tau,
\end{equation}
with $0<\beta<1/\tau$.
 {Notably, our semi-implicit scheme treats the quadratic energy for prestrained plates implicitly, while handling the non-convex constraint explicitly through the tangent space update $\delta\vy^{n+1}\in\mathcal{F}(\vy^n)$. This approach ensures that a linear system is solved at each iteration.}

By canceling $\vw^n$ in \eqref{eq:nesterov_scheme} and using that $a_g$ is a bilinear form, we derive 
\begin{equation*}
\frac{1}{\tau^{2}}(\delta\vy^{n+1},\vv)_{H^2} + a_g(\delta\vy^{n+1},\vv)=-a_g(\vy^n,\vv)-\eta^{n}a_g(\delta\vy^n,\vv)+\frac{\eta^{n}}{\tau^{2}}(\delta\vy^n,\vv)_{H^2},
\end{equation*} 
and then minus $\frac{1}{\tau^{2}}(\delta\vy^{n},\vv)_{H^2}$ on both sides and rearrange the terms to get 
\begin{align}\label{eq:nesterov_scheme_2}
\frac{1}{\tau^{2}}(\delta\vy^{n+1}-\delta\vy^{n},\vv)_{H^2}&+\frac{1-\eta^{n}}{\tau^{2}}(\delta\vy^n,\vv)_{H^2}+a_g(\vy^n,\vv) \\ \nonumber
&+\eta^{n}a_g(\delta\vy^n,\vv)+a_g(\delta\vy^{n+1},\vv)=0.
\end{align} 
Considering a continuous flow $Y(t,\vx)\in C^2(0,\infty;[H^2(\Omega)]^3)$, we set $\vy^n:=Y(t_n,\cdot)$ with an arbitrary fixed $t_n>0$, $n:=t_n/\tau-\alpha+1$ ($\tau$ is properly chosen to guarantee $n$ is an integer) and $t_{n\pm1}:=t_n\pm\tau$.
Using the Taylor expansion of $Y(t_{n+1},\cdot)$ and $Y(t_{n-1},\cdot)$ around $t=t_n$ in \eqref{eq:nesterov_scheme_2}, noticing that $\frac{1-\eta^{n}}{\tau^{2}}=\frac{\alpha}{\tau t_n}$ when $\eta^n=\frac{n-1}{n+\alpha-1}$, and dropping $\mathcal{O}(\tau)$ terms for small enough $\tau$ yield the weak formulation of following PDE dynamics with artificial time $t=t_n$: 
\begin{equation}\label{eq:ODE}
\ddot{Y}(t,\cdot)+\frac{\alpha}{t}\dot{Y}(t,\cdot)+\delta E_{pre}[Y(t,\cdot)] = 0.
\end{equation} 
We note that in this process the second line of \eqref{eq:nesterov_scheme_2} vanishes, because $|\eta^n|\le1$ and 
$\delta\vy^n,\delta\vy^{n+1}=\mathcal{O}(\tau)$.
Separately, we can repeat the same formal asymptotic process in the constraint $\delta\vy^{n+1}\in\mathcal{F}(\vy^n)$, i.e., $(\nabla\delta\vy^{n+1})^T\nabla\vy^n+(\nabla\vy^n)^T\nabla\delta\vy^{n+1}=0$, and drop the $\mathcal{O}(\tau^2)$ terms to obtain $\nabla\dot{Y}^T\nabla Y+\nabla Y^T\nabla\dot{Y}=0$, namely $\frac{d(\nabla Y^T\nabla Y)}{dt}=0$, which shows consistency to the metric constraint given that $\nabla Y(0,\vx)^T\nabla Y(0,\vx)=g(\vx)$.
Moreover, when $\eta^n=1-\beta\tau$, the damping coefficient $\frac{\alpha}{t}$ in \eqref{eq:ODE} is then replaced by a constant $\beta$, which corresponds to the heavy-ball dynamics as discussed in \cite{AttGouRed2000}.  {The proposed method is summarized later in Section \ref{sec:alg} as Algorithm \ref{algo:acc_flow}}

\subsection{ {Energy stability and admissibility}}\label{sec:property-prestrained}
A direct calculation implies that the PDE \eqref{eq:ODE} guarantees the monotone decreasing of $\frac12\|\dot{Y}(t,\cdot)\|^2+E_{pre}[Y(t,\cdot)]$ along the flow $Y(t,\cdot)$. This motivates the following theorem.  

\begin{thm}[Total energy stability]\label{thm:energy-stab-second-pre}
Suppose $n\ge1$, $\vy^n\in[H^2(\Omega)]^3$ and $\delta\vy^n\in[H^2_0(\Omega)]^3$ are given, and $\delta\vy^{n+1}$ is solved by \eqref{eq:nesterov_scheme_2}. Define $\vy^{n+1}:=\vy^n+\delta\vy^{n+1}$. If $\tau\le C(\lambda,\mu,g)^{-1/2}$ with the constant $C(\lambda,\mu,g)$ appearing in \eqref{eq:coer-ag}, then the following energy stability estimate is valid: 
\begin{equation}\label{eq:energy-stab-pre-1}
E_{pre}[\vy^{n+1}]+\frac{1}{2\tau^{2}}|\delta\vy^{n+1}|_{H^2}^2+\frac{1-\eta^{n}}{2\tau^{2}}|\delta\vy^{n+1}|_{H^2}^2\le E_{pre}[\vy^{n}]+\frac{1}{2\tau^{2}}|\delta\vy^{n}|_{H^2}^2.
\end{equation}
\end{thm}

\begin{proof}
Taking $\vv=\delta\vy^{n+1}$ in \eqref{eq:nesterov_scheme_2}, exploiting the bilinearity of $a_g$ and $(\cdot,\cdot)_H^2$, and employing the algebraic identity $a(a-b)=\frac12a^2-\frac12b^2+\frac12(a-b)^2$ and $b(a-b)=\frac12a^2-\frac12b^2-\frac12(a-b)^2$, 
we derive 
{\small
\begin{align}\label{eq:energy-stab-pre-inter}
0=&\frac{1}{2\tau^{2}}\Big(|\delta\vy^{n+1}|_{H^2}^2-|\delta\vy^{n}|_{H^2}^2\Big)+E_{pre}[\vy^{n+1}]-E_{pre}[\vy^{n}] \\ \nonumber
&+\frac{1-\eta^{n}}{2\tau^{2}}\Big(|\delta\vy^n|_{H^2}^2+|\delta\vy^{n+1}|_{H^2}^2\Big)+\frac{1+\eta^n}{2}a_g(\delta\vy^{n+1},\delta\vy^{n+1})+\frac{\eta^n}{2}a_g(\delta\vy^n,\delta\vy^{n})\\ \nonumber
&+\frac{\eta^{n}}{2\tau^{2}}|\delta\vy^{n+1}-\delta\vy^{n}|_{H^2}^2-\frac{\eta^n}{2}a_g(\delta\vy^{n+1}-\delta\vy^n,\delta\vy^{n+1}-\delta\vy^n).
\end{align}  
}
Note that the first line gives the desired discrete total energy at $\vy^{n+1}$ and $\vy^{n}$, and each term in the second line is non-negative as $0\le\eta^n<1$ ($n\ge1$). The third line is also non-negative after incorporating the coercivity estimate \eqref{eq:coer-ag} and the assumption $\tau<C(\lambda,\mu,g)^{-1/2}$. This validates the energy stability \eqref{eq:energy-stab-pre-1}.
\end{proof}

Estimate \eqref{eq:energy-stab-pre-1} implies that the energy stability of the total energy $E_{pre}[\vy^{n}]+\frac{1}{2\tau^{2}}|\delta\vy^{n}|_{H^2}^2$, where the second term approximates $\frac12\|\dot{Y}\|_{H^2}^2$ and can be viewed as a kinetic energy. The term $\frac{1-\eta^{n}}{2\tau^{2}}|\delta\vy^{n+1}|_{H^2}^2$ in \eqref{eq:energy-stab-pre-1} is an energy dissipation term, which has other options, according to \eqref{eq:energy-stab-pre-inter}.  We emphasize that the constant in the condition $\tau< C(\lambda,\mu,g)^{-1/2}$ is not affected by any discretization parameter and depends solely on the problem data. This condition is quite mild when the SPD matrix-valued $g$ is bounded away from singular.    

Moreover, as $\vw^0=\vy^0$, taking $n=0$ in \eqref{eq:nesterov_scheme} gives us a step of gradient flow with time-step $\tau^2$, and therefore we have the energy stability estimate \cite{bonito2022ldg}
\begin{equation}\label{eq:energy-stability-step1-gf}
E_{pre}[\vy^{1}]+\frac{1}{\tau^{2}}|\delta\vy^{1}|_{H^2}^2\le E_{pre}[\vy^{0}].
\end{equation} 
Consequently, summing \eqref{eq:energy-stab-pre-1} over $n=1$ to $N$ for $N\ge1$ yields 
\begin{equation}\label{eq:energy-stability-sum-final}
E_{pre}[\vy^{N+1}]+\frac{1}{2\tau^{2}}|\delta\vy^{N+1}|_{H^2}^2+\sum_{n=1}^N\frac{1-\eta^n}{2\tau^{2}}|\delta\vy^{n+1}|_{H^2}^2\le E_{pre}[\vy^{0}].
\end{equation}
Note that the right-hand side  (RHS) of \eqref{eq:energy-stability-sum-final} can still be considered as a total energy if one imposes zero velocity initial condition $\dot{Y}(0,\cdot)=0$ in \eqref{eq:ODE}. 

Due to the monotone decay of total energy, as stated in Theorem \ref{thm:energy-stab-second-pre}, we stop the iterations when the total energy stabilizes; see Algorithm \ref{algo:acc_flow}.   

Based on the energy stability result, the limiting behavior of the velocity term can be immediately derived, which is stated below. 
\begin{cor}\label{cor:kinetic-lim}
Let $\{\vy^{n}\}_{n\in\mathbb{N}}$ be the sequence produced by \eqref{eq:nesterov_scheme}-\eqref{eq:update_w} with a fixed $\tau>0$, then the increments $\delta\vy^{n+1}:=\vy^{n+1}-\vy^n$ satisfy that $\lim_{n\to\infty}|\delta\vy^{n}|_{H^2}=0$.
\end{cor}
\begin{proof}
Summing \eqref{eq:energy-stab-pre-inter} over $n$ and conducting telescopic cancellation, we derive that $\sum_{n=1}^{\infty}(\eta^n/2)a_g(\delta\vy^{n},\delta\vy^{n})$ is bounded. We further note that $0\le\eta^n<1$ and recall the coercivity \eqref{eq:coer-ag} to conclude that $\lim_{n\to\infty}|\delta\vy^{n}|_{H^2}=0$.  
\end{proof}

Moreover, following the framework in \cite[Corollary 3.4]{dong2024bdf}, under realistic assumptions, we can show that a subsequence of $\vy^{n}$ converges to a limit $\vy^*$, which is a minimizer of $E_{pre}$ restricting to the tangent plane $\mathcal{F}(\vy^*)$. 

The metric constraint $\I[\vy]=g$ is not imposed exactly in the flow because we only require $\delta\vy^{n+1}\in\mathcal{F}(\vy^n)$. Next, we provide an estimate on the control of the violation of metric constraint to justify the admissibility of $\vy^{N+1}$ in $\mathcal{A}^{\epsilon}_{g,1}$ for any $N\ge0$.
 
\begin{thm}[Admissibility]\label{thm:admissibility-pre}
Let $\vy^0\in\mathcal{A}_{g}$ and $\{\vy^{n}\}_{n\in\mathbb{N}}$ be the sequence produced by \eqref{eq:nesterov_scheme}-\eqref{eq:update_w} with a fixed $\tau\le  C(\lambda,\mu,g)^{-1/2}$, where $ C(\lambda,\mu,g)$ is the constant appearing in \eqref{eq:coer-ag}. For any $N\ge1$, $\vy^{N+1}\in\mathcal{A}^{\epsilon}_{g,1}$ with $\epsilon=C(\Omega,\alpha)N\tau^2E_{pre}[\vy^{0}]$ when $\eta^n$ is chosen as \eqref{eq:eta-nesterov} and $\epsilon=C(\Omega,\beta)\tau E_{pre}[\mathbf{y}^{0}]$ when $\eta^n$ is chosen as \eqref{eq:eta-heavy-ball}.
\end{thm}
\begin{proof}
Using the tangent plane constraint $\delta\vy_h^{n+1}\in\mathcal{F}(\vy_h^n)$, a Poincar\'e-type inequality for $\delta\vy_h^{n+1}\in[H^2_0(\Omega)]^3$ and telescopic cancellation,   we derive 
\begin{equation}\label{eq:cons-vio-ortho}
\|\I[\vy^{N+1}]-g\|_{L^1(\Omega)}\le\|\I[\vy^0]-g\|_{L^1(\Omega)} +C(\Omega)\sum_{n=0}^N|\delta\vy_h^{n+1}|_{H^2}^2.
\end{equation}   
When $\eta^n$ is chosen as \eqref{eq:eta-nesterov}, for $N\ge1$ we resort to \eqref{eq:energy-stability-sum-final} and compute
\begin{equation}\label{eq:estimate-inc}
\frac{\alpha}{2(N+\alpha-1)\tau^{2}}\sum_{n=1}^N|\delta\vy^{n+1}|_{H^2}^2<\sum_{n=1}^N\frac{\alpha}{2(n+\alpha-1)\tau^{2}}|\delta\vy^{n+1}|_{H^2}^2\le E_{pre}[\vy^{0}],
\end{equation}
and together with \eqref{eq:energy-stability-step1-gf} we conclude that for $N\ge1$
\begin{equation*}\label{eq:iso-violation}
\|\I[\vy^{N+1}]-g\|_{L^1(\Omega)}\le C(\Omega)(1+\frac{2(N+\alpha-1)}{\alpha})\tau^2E_{pre}[\vy^{0}]\le C(\Omega,\alpha)N\tau^2E_{pre}[\vy^{0}].
\end{equation*}   
When $\eta^n$ is chosen as \eqref{eq:eta-heavy-ball}, it is straightforward to obtain the desired estimate upon applying an estimate similar to \eqref{eq:estimate-inc} with $\eta^n=1-\beta\tau$.
\end{proof}

\begin{remark} \label{rmk:control-violation}
At a first glance, the constraint violation given in Theorem \ref{thm:admissibility-pre} for the case when $\eta^n$ is given by \eqref{eq:eta-nesterov} may seem problematic, as the control parameter $\epsilon$ is not uniform and depends on iteration numbers. However, by employing the stopping criteria with respect to total energy, it is consistently observed that the total number of iterations $N=\mathcal{O}(\tau^{-1})$ in numerical simulations in Section \ref{sec:numerics} with fixed tolerance. Specifically, recalling $E_{total}[\mathbf{y}^{n+1}] := E_{pre}[\mathbf{y}^{n+1}] + \frac{1}{2\tau^{2}}|\delta\mathbf{y}^{n+1}|_{H^2}^2$, and given $\tol > 0$, the iteration is terminated at the $N$-th step if 
\begin{equation}\label{eq:stop-criteria}
\tau^{-1}(E_{total}[\mathbf{y}^{N-1}] - E_{total}[\mathbf{y}^{N}]) < \tol.
\end{equation}

We resort to the continuous second-order dynamics in \eqref{eq:ODE} for a formal explanation of $N=\mathcal{O}(\tau^{-1})$. Since our scheme can be asymptotically viewed as a temporal discretization of \eqref{eq:ODE}, we can establish a connection between the iteration number and the terminal time $T$ of the continuous dynamics as $T = N\tau$. On the continuous level, the practical stopping criteria \eqref{eq:stop-criteria} can be interpreted as 
\begin{equation}\label{eq:stop-criteria-cont}
\Big|\frac{d}{dt}(\frac12\|\dot{Y}\|^2+E_{pre}[Y])\Big| < \tol.
\end{equation}
It is evident from \eqref{eq:ODE} that $\lim_{t\to\infty}\frac{d}{dt}(\frac12\|\dot{Y}\|^2+E_{pre}[Y])=0$, ensuring the existence of $T(\tol) > 0$, depending on the fixed $\tol > 0$, such that \eqref{eq:stop-criteria-cont} is satisfied. Consequently, on the discrete level, $N=T(\tol)\tau^{-1}=\mathcal{O}(\tau^{-1})$.

Returning to Theorem \ref{thm:admissibility-pre}, we can refine the estimate of constraint violation so that $\mathbf{y}^{N} \in \mathcal{A}^{\epsilon}_{g,1}$ with $\epsilon=C(\Omega,\alpha,\tol)\tau E_{pre}[\mathbf{y}^{0}]$. 
In Section \ref{sec:numerics}, all numerical experiments confirm that the constraint violation is of order $\mathcal{O}(\tau)$ and the constant $C(\Omega,\alpha,\tol)$ stabilizes as $\tol\to0$ (so does the constraint violation).   
\end{remark}

\subsection{Backtracking and restarting}\label{sec:backtrack}
It is important to note that Algorithm \ref{algo:acc_flow} cannot guarantee a monotonic decrease of potential energies $E=E_{pre}$. 
However, we can further adopt a backtracking and restarting strategy whenever the potential energy increases.  {Specifically, when $E_{pre}[\vy^{n+1}] > E_{pre}[\vy^{n}]$, we reject the update $\vy^{n+1}$, revert to $\vy^n$, and recompute a new update via \eqref{eq:nesterov_scheme} with $\eta^{n+1}$ reinitialized with the choice \eqref{eq:eta-nesterov} and $n=1$.}  

The update immediately after the backtracking and restarting gives us $\eta^{n+1}=0$ and $\vw^{n+1}=\vy^{n+1}=\vy^n$, which further implies that the immediate next step with \eqref{eq:nesterov_scheme} is a gradient flow update with time step $\tau^2$. This 
guarantees a monotone decay of the potential energy $E=E_{pre}$.

On the continuous level, i.e., referring to \eqref{eq:ODE}, such a backtracking and restarting scheme could correspond to discontinuous damping coefficients, which brings more challenges to analysis.
Additionally, we emphasize that this strategy does not apply to the scenario of $\eta^n=1-\beta\tau$, as $\eta^n$ remains constant in the flow. 
A potential adjustment may involve modifying $\beta$ in the flow to use a gradient flow step alternatively for reducing $E$. We defer the exploration of such a time-dependent $\beta$ to future studies.

 {The modified algorithm is presented in Algorithm \ref{algo:acc_flow_backtrack}.}

\subsection{ {BDF2 temporal discretization}}\label{sec:BDF2}
Inspired by \cite{akrivis2023quadratic} for gradient flows, we modify \eqref{eq:nesterov_scheme} using a BDF2 approximation to time derivative:
\[\dot{\vy}^n\approx\frac{3\vy^n-4\vy^{n-1}+\vy^{n-2}}{2\tau}.\] 
This modification leads to a higher order approximation of constraints in \cite{akrivis2023quadratic}, and we expect to achieve similar improvements computationally for accelerated flows. The new scheme is described in Algorithm \ref{algo:acc_flow_BDF2}.
It is crucial to compute increments in the tangent space of an extrapolation $\delta\vy^{n+1}\in\mathcal{F}(2\vy^n-\vy^{n-1})$ such that  
\begin{equation}\label{eq:nesterov_scheme_quad_cons}
\tau^{-2}(\delta\vy^{n+1},\vv)_{H^2} + \frac23a_g(\delta\vy^{n+1},\vv)=a_g(-\frac43\vw^n+\frac13\vw^{n-1},\vv)+\tau^{-2}(\vw^n-\vy^{n},\vv)_{H^2}.
\end{equation}
The numerical performance of this algorithm (cf. Algorithm \ref{algo:acc_flow_BDF2}) is reported in Section \ref{sec:numerics}. In our accompanying paper \cite{dong2024bdf} we establish estimates on constraint violations of BDF2-BDF4 scheme for accelerated flows with tangent space update.    

\section{ {Implicit-Explicit accelerated gradient flows}}\label{sec:bilayer}
 {We next propose implicit~-~explicit (IMEX) accelerated flows for bilayer plates \eqref{bilayer-energy-3}, where the non-convex part of the energy $E^{nc}_{bi}$ is treated explicitly while the convex part remains implicit. This scheme extends classical IMEX methods by combining acceleration flows with the tangent space update strategy as in Section~\ref{sec:prestrained}, which linearizes the isometry constraint $\nabla\vy^T\nabla\vy~=~I_2$.
We structure this section similarly as Section \ref{sec:acc-flow}.}
 
\subsection{ {BDF1 iterative scheme}}\label{sec:schemes-bi} 
Given $\vy^0\in\mathcal{A}$ that satisfies the constraint and boundary condition as in \eqref{iso-admissible} and setting $\vw^0=\vy^0$, we compute the increment $\delta\vy^{n+1}\in\mathcal{F}(\vy^n)$ in each iteration and update by
\begin{align}\label{eq:bilayer_nesterov_scheme}
\tau^{-2}(\delta\vy^{n+1},\vv)_{H^2} + a(\delta\vy^{n+1},\vv)&=\ell[\vw^n](\vv)-a(\vw^n,\vv)+\tau^{-2}(\vw^n-\vy^{n},\vv)_{H^2} \\ \label{eq:update_y_bi}
\vy^{n+1}&=\vy^{n}+\delta\vy^{n+1} \\ \label{eq:update_w_bi}
\vw^{n+1}&=\vy^{n+1}+\eta^{n+1}\delta\vy^{n+1},
\end{align}
where $\eta^n=\frac{n-1}{n+\alpha -1}$ or $\eta^n=1-\beta\tau$. Here, $a(\vy,\vv):=\int_{\Omega}D^2\vy:D^2\vv$, and we define $\ell[\vy](\cdot):=\delta E_{bi}^{nc}[\vy](\cdot)$ as 
\begin{align}\label{eq:def-l} 
\ell[\vy](\vv):=&\sum_{i,j=1}^2\int_{\Omega}\Big(\partial_{ij}\vv\cdot(\partial_1\vy\times\partial_2\vy)+\partial_{ij}\vy\cdot(\partial_1\vv\times\partial_2\vy+\partial_1\vy\times\partial_2\vv)\Big)Z_{ij}.
\end{align}

 {Comparing to the prestrained plate problem, the key distinction lies in the incorporation of $\ell[\vw^n](\vv)$, which handles the non-convex energy component $E_{bi}^{nc}$ explicitly. As for the remaining aspects, this implicit-explicit accelerated flow method for bilayer plates formally shares the same structure as the semi-implicit one for prestrained plates as presented in Algorithm~\ref{algo:acc_flow}.} 

\subsection{ {Energy stability and admissibility}}\label{sec:property-bi}
By Sobolev embedding of $H^1$ into $L^4$ and H\"older inequality, for any $\vv\in[H^2_0(\Omega)]^3$ the following estimate is valid: 
\begin{equation}\label{eq:est-ln}
|\ell[\vy](\vv)|\le C(\Omega,Z)|\vv|_{H^2}\|\vy\|_{H^2}^2.
\end{equation}
Similar to \eqref{eq:nesterov_scheme_2}, we rewrite \eqref{eq:bilayer_nesterov_scheme} as 
\begin{align}\label{eq:bilayer_nesterov_scheme_2}
\frac{1}{\tau^{2}}(\delta\vy^{n+1}-\delta\vy^{n},\vv)_{H^2}&+\frac{1-\eta^{n}}{\tau^{2}}(\delta\vy^n,\vv)_{H^2}+a(\vy^n,\vv)-\ell[\vw^n](\vv) \\ \nonumber
&+\eta^{n}a(\delta\vy^n,\vv)+a(\delta\vy^{n+1},\vv)=0.
\end{align} 
We can further replace $\ell[\vw^n](\vv)$ in \eqref{eq:bilayer_nesterov_scheme_2} by $\ell[\vy^n](\vv)$ after introducing the remainder $R^n(\vv)$, which is defined as   
\begin{equation}\label{def:Rn}
R^n(\vv):=\ell[\vw^n](\vv)-\ell[\vy^n](\vv).
\end{equation}
Substituting back to \eqref{eq:bilayer_nesterov_scheme_2}, we get 
\begin{align}\label{eq:bilayer_nesterov_scheme_3}
\frac{1}{\tau^{2}}(\delta\vy^{n+1}-\delta\vy^{n},\vv)_{H^2}&+\frac{1-\eta^{n}}{\tau^{2}}(\delta\vy^n,\vv)_{H^2}+a(\vy^n,\vv)-\ell[\vy^n](\vv)\\ \nonumber
&+\eta^{n}a(\delta\vy^n,\vv)+a(\delta\vy^{n+1},\vv)=R^n(\vv).
\end{align}

It is worth noting that $R^n(\vv)$ has an upper bound as follows. 
\begin{lemma}[Estimate of remainder]\label{lem:est-Rn}
The following estimate holds for $R^n(\vv)$ defined by \eqref{def:Rn} and any $\vv\in[H^2_0(\Omega)]^3$: 
\begin{equation}\label{eq:est-Rnh}
|R^n(\vv)|\le C(\Omega,Z)\Big(|\vv|_{H^2}|\delta\vy^n|_{H^2}\Big(|\vy^{n}|_{H^2}+C(\vvarphi,\Phi)\Big)+|\vv|_{H^2}|\delta\vy^n|^2_{H^2}\Big).
\end{equation}
\end{lemma}
\begin{proof}
Substituting the definition of $\vw^n$ in \eqref{eq:update_w_bi} into the expression \eqref{eq:def-l} of linear functional $\ell$ and calculating the difference in \eqref{def:Rn},
we can explicitly express $R^n$ as a sum of several integrals involving $\vv,\vy^n,\delta\vy^n$. Here, we omit the tedious expression, while the calculation is straightforward.  
We then estimate each term by using H\"older inequality, Poincar\'e inequality, the embedding of $H^1(\Omega)$ into $L^4(\Omega)$, and the fact that $|\eta^n|<1$, to arrive at \eqref{eq:est-Rnh}.
\end{proof}

The energy $E_{bi}$ is obviously coercive when $\vy$ satisfies $\nabla\vy^T\nabla\vy=I_2$ exactly. However, since the isometry constraint is not kept in our schemes, we need to show coercivity under a relaxed constraint for $\vy\in\A^{\epsilon}_{I_2,2}$. Here, the use of $L^2$ norm in $\A^{\epsilon}_{I_2,2}$ is critical to estimate the cubic term $E^{nc}_{bi}$; $L^1$ norm seems to be not enough.
\begin{lemma}[Coercivity]\label{lem:coercivity-bi}
Let $\epsilon>0$ and $\vy\in\A^{\epsilon}_{I_2,2}$. There exists a positive constant $C(\Omega)$
such that 
\begin{equation}\label{eq:coercivity-bi}
\|\vy\|_{H^2}^2\le C(\Omega)\Big(E_{bi}[\vy]
+C(\vvarphi,\Phi)+\big(\epsilon^2+2|\Omega|\big)\|Z\|_{L^{\infty}}^2\Big).
\end{equation}
\end{lemma}
\begin{proof}
We first note that 
\begin{align}\label{eq:est-L4-grady}
\|\nabla\vy\|_{L^4}^4&\le2\|\nabla\vy^T\nabla\vy\|_{L^2}^2\le4\big(\|\nabla\vy^T\nabla\vy-I_2\|_{L^2}^2+\|I_2\|_{L^2}^2\big)=4\big(\epsilon^2+2|\Omega|\big).
\end{align}
Using \eqref{eq:est-L4-grady}, H\"older inequality and Cauchy inequality to estimate $E^{nc}_{bi}[\vy]$, we obtain  
\begin{align}\label{eq:coercivity-bi-1}
|\vy|_{H^2}^2 \le 2E_{bi}[\vy]+32\big(\epsilon^2+2|\Omega|\big)\|Z\|_{L^{\infty}}^2+\frac{1}{2}|\vy|_{H^2}^2.
\end{align}
Then, incorporating the Poincar\'e inequality $\|\vy\|_{H^2}^2\le C(\Omega)(|\vy|_{H^2}^2+C(\vvarphi,\Phi))$ we conclude the proof. 
\end{proof}

Note that as $\vw^0=\vy^0$, taking $n=0$ in \eqref{eq:bilayer_nesterov_scheme} provides a step of gradient flow with the time-step $\tau^2$, and therefore we have the energy stability estimate \cite{bonito2023gamma}
\begin{equation}\label{eq:energy-stability-step1-bi}
E_{bi}[\vy^{1}]+\frac{1}{2\tau^{2}}|\delta\vy^{1}|_{H^2}^2\le E_{bi}[\vy^{0}].
\end{equation} 

By exploiting the estimate of remainder in Lemma \ref{lem:est-Rn}, the coercivity in Lemma \ref{lem:coercivity-bi}, and considering an induction argument, we establish the total energy stability as follows for the semi-implicit accelerated gradient flow scheme of bilayer plates. The rather tedious technical details of proof is given in the Appendix. 

\begin{thm}[Total energy stability for bilayer plates]\label{thm:energy-stab-second-bi}
Let $\vy^0\in\A$ and $M=1,\ldots,N-1$ be the iteration number of accelerated flow defined by \eqref{eq:bilayer_nesterov_scheme}-\eqref{eq:update_w_bi}. We further assume the total number of iterations $N=\alpha_0\tau^{-1}$ (as discussed in Remark \ref{rmk:control-violation}) with a constant $\alpha_0>0$ when $\eta^n=\frac{n-1}{n+\alpha-1}$. 
There exists a constant $\alpha_1=\alpha_1(\vy^0,\vvarphi,\Phi,Z,\Omega,\alpha,\alpha_0)>0$ independent of $M$ such that if $\tau<\alpha_1$, then the following energy stability estimate is valid: 
{\footnotesize
\begin{equation}\label{eq:bilayer-second-energy-onestep}
\frac{1}{2\tau^{2}}|\delta\vy^{M+1}|_{H^2}^2+E_{bi}[\vy^{M+1}]+\frac{1-\eta^{M}}{4\tau^{2}}\Big(|\delta\vy^{M+1}|_{H^2}^2+|\delta\vy^{M}|_{H^2}^2\Big)\le \frac{1}{2\tau^{2}}|\delta\vy^{M}|_{H^2}^2+E_{bi}[\vy^{M}].
\end{equation}
}
In addition, there exists $\alpha_2=\alpha_2(\Omega,\alpha,\alpha_0)>0$, independent of $M$, such that 
\begin{equation}\label{eq:cons-control-bi}
D_{I_2,2}[\vy^{M+1}]\le\alpha_2\tau E_{bi}[\vy^0].
\end{equation}
\end{thm}

Moreover, analogous to Corollary \ref{cor:kinetic-lim} for prestrained plates, we conclude that $\lim_{n\to\infty}|\delta\vy^{n}|_{H^2}=0$ for bilayer plates, implying that the kinetic energy converges to $0$ as well.

\subsection{ {Backtracking and BDF2 discretization}}\label{sec:variants_bi} 
 {Variants of algorithms for bilayer plates are analogous to the methods developed in Sections~\ref{sec:backtrack} and~\ref{sec:BDF2}.
For the BDF2 scheme, we compute $\delta\vy^{n+1}\in\mathcal{F}(2\vy^n-\vy^{n-1})$ satisfying  
\begin{align}\label{eq:nesterov_scheme_quad_cons_bi}
&\tau^{-2}(\delta\vy^{n+1},\vv)_{H^2} + \frac{2}{3}a(\delta\vy^{n+1},\vv) \\ \nonumber
&=\ell[\vw^n](\vv)+a\left(-\frac{4}{3}\vw^n+\frac{1}{3}\vw^{n-1},\vv\right)+\tau^{-2}(\vw^n-\vy^{n},\vv)_{H^2},
\end{align}
where $\ell[\vw^n](\vv)$ is the only modification compared to \eqref{eq:nesterov_scheme_quad_cons}. 
The remaining methodology follows identically, with implementations detailed in Algorithms~\ref{algo:acc_flow_backtrack} and~\ref{algo:acc_flow_BDF2}.
}

\section{ {Algorithms and spatial discretization}}\label{sec:alg}
\subsection{ {Summary of algorithms}}\label{sec:sum-alg}
 {All the algorithms proposed in the previous sections are collected here in a nutshell.
Denoting 
\begin{equation}\label{eq:E-total}
    E_{total}[\mathbf{y}^{n+1}] := E[\mathbf{y}^{n+1}] + \frac{1}{2\tau^{2}}|\delta\mathbf{y}^{n+1}|_{H^2}^2,
\end{equation}
with $E=E_{pre}$ or $E_{bi}$, Algorithm \ref{algo:acc_flow} summarizes the working procedure of the accelerated flows in combination with tangent space update strategy for prestrained or bilayer plates, where the choice of $\eta^n$ is either \eqref{eq:eta-nesterov} or \eqref{eq:eta-heavy-ball}.} 
 
\begin{algorithm}
\caption{ {Semi-Implicit/Implicit-Explicit} accelerated gradient flows}
\label{algo:acc_flow}
\begin{algorithmic}[1]
\STATE{Choose a pseudo time-step $\tau>0$, a tolerance constant tol and parameter $\alpha\ge3$ or $0<\beta\le1/\tau$}
\STATE{Choose $\vy^0\in\A_g$ (\textbf{prestrained plates}) or  $\vy^0\in\A$ (\textbf{bilayer plates}), set $\vw^0=\vy^0$, $n=0$ and $\Delta E_{total}=\infty$}
\WHILE{$\tau^{-1}\Delta E_{total}>$ tol}
\STATE{\textbf{Solve} \eqref{eq:nesterov_scheme} ({\bf prestrained plates}) or \eqref{eq:bilayer_nesterov_scheme} ({\bf bilayer plates}) for $\delta\vy^{n+1}\in\mathcal{F}(\vy^n)$}\\
\STATE{\textbf{Update} $\vy^{n+1} = \vy^{n}+\delta\vy^{n+1}$, $\eta^{n+1}=\frac{n}{n+\alpha}$ or $\eta^{n+1}\equiv1-\beta\tau$, \\  $\vw^{n+1}=\vy^{n+1}+\eta^{n+1}\delta\vy^{n+1}$
}
\STATE{\textbf{Set} $\Delta E_{total}=E_{total}[\vy^{n+1}]-E_{total}[\vy^n]$ and $n=n+1$}
\ENDWHILE
\RETURN $\vy^n$
\end{algorithmic}
\end{algorithm}

 {Furthermore, Algorithm~\ref{algo:acc_flow_backtrack} presents the semi-implicit/implicit-explicit accelerated flows with the backtracking strategy from Sections~\ref{sec:backtrack} and~\ref{sec:variants_bi}, while Algorithm~\ref{algo:acc_flow_BDF2} incorporates the BDF2 temporal discretization introduced in Sections~\ref{sec:BDF2} and~\ref{sec:variants_bi}.}

\begin{algorithm}
\caption{ {Accelerated flows with backtracking}}
\label{algo:acc_flow_backtrack}
\begin{algorithmic}[1]
\STATE{Given a pseudo time-step $\tau>0$, a tolerance $tol$ and parameter $\alpha\ge3$}
\STATE{Choose $\vy^0\in\A_g$ (\textbf{prestrained plates}) or  $\vy^0\in\A$ (\textbf{bilayer plates}), set $\vw^0=\vy^0$, $n=0$, $k=0$ and $\Delta E_{total}=\infty$}
\WHILE{$\tau^{-1}\Delta E_{total}>$tol}
\STATE{\textbf{Solve} \eqref{eq:nesterov_scheme} (\textbf{prestrained plates}) or \eqref{eq:bilayer_nesterov_scheme} (\textbf{bilayer plates}) for $\delta\vy^{n+1}\in\mathcal{F}(\vy^n)$}\\
\STATE{\textbf{Update} $\vy^{n+1} = \vy^{n}+\delta\vy^{n+1}$}
\STATE{ \textbf{Set} $\Delta E_{total}=E_{total}[\vy^{n+1}]-E_{total}[\vy^n]$ and $\Delta E=E[\vy^{n+1}]-E[\vy^n]$ }
\IF{$\Delta E<0$} 
\STATE{\textbf{Set}  $k=k+1$}
\ELSE{}
\STATE{\textbf{Backtrack} $\vy^{n+1}=\vy^{n}$}
\STATE{\textbf{Restart} $k=1$}
\ENDIF
\STATE{\textbf{Update} }
  $\eta^{n+1}=\frac{k-1}{k-1+\alpha}$ and $\vw^{n+1}=\vy^{n+1}+\eta^{n+1}\delta\vy^{n+1}$
\STATE{ \textbf{Set}  $n=n+1$}  
\ENDWHILE
\RETURN $\vy^n$
\end{algorithmic}
\end{algorithm}

\begin{algorithm}
\caption{ {Accelerated flows with BDF2 discretization}}
\label{algo:acc_flow_BDF2}
\begin{algorithmic}[1]
\STATE{Given a pseudo time-step $\tau>0$, a tolerance $tol$ and parameter $\alpha\ge3$}
\STATE{Choose $\vy^0\in\A_g$ (\textbf{prestrained plates}) or  $\vy^0\in\A$ (\textbf{bilayer plates}),   set $\vw^0=\vy^0$, $n=0$, $k=0$ and $\Delta E_{total}=\infty$}
\STATE{Compute $\delta\vy^{1}\in\mathcal{F}(\vy^0)$ by \eqref{eq:nesterov_scheme} (\textbf{prestrained plates}) or \eqref{eq:bilayer_nesterov_scheme} (\textbf{bilayer plates})}
\STATE{Update $\vy^1=\vy^0+\delta\vy^1$,  $\vw^1=\vy^1$ and $n=n+1$}
\WHILE{$\tau^{-1}\Delta E_{total}>$tol}
\STATE{\textbf{Solve} \eqref{eq:nesterov_scheme_quad_cons} (\textbf{prestrained plates}) or \eqref{eq:nesterov_scheme_quad_cons_bi} (\textbf{bilayer plates}) \\
for $\delta\vy^{n+1}\in\mathcal{F}(2\vy^n-\vy^{n-1})$}
\STATE{\textbf{Update} $\vy^{n+1}=\frac43\vy^{n}-\frac13\vy^{n-1}+\frac23\delta\vy^{n+1}$}
\STATE{ \textbf{Set} $\Delta E_{total}=E_{total}[\vy^{n+1}]-E_{total}[\vy^n]$ and $\Delta E=E[\vy^{n+1}]-E[\vy^n]$ }
\IF{$\Delta E<0$} 
\STATE{\textbf{Set}  $k=k+1$}
\ELSE{}
\STATE{\textbf{Backtrack} $\vy^{n+1}=\vy^{n}$}
\STATE{\textbf{Restart} $k=1$}
\ENDIF
\STATE{\textbf{Update} }
  $\eta^{n+1}=\frac{k-1}{k-1+\alpha}$ and $\vw^{n+1}=\vy^{n+1}+\eta^{n+1}\delta\vy^{n+1}$
\STATE{ \textbf{Set}  $n=n+1$} 
\ENDWHILE
\RETURN $\vy^n$
\end{algorithmic}
\end{algorithm} 

\subsection{ {Spatial discretization}}\label{sec:sp-dis}
 {In the following we provide a concise description of the spatial discretization employed in our numerical simulations.}

We utilize Morley FEM \cite{morley1968triangular} to discretize the proposed accelerated gradient flows spatially. 
However, the semi-discrete schemes presented in Algorithms \ref{algo:acc_flow}, \ref{algo:acc_flow_backtrack}  and \ref{algo:acc_flow_BDF2} should work well along with any other suitable FEMs for spatial discretization, such as Kirchhoff FEM, IPDG, and LDG implemented in previous studies on nonlinear plates models \eqref{bilayer-energy-3} and \eqref{prob:min_Eg}.

We consider a shape-regular triangulation $\Th$ of $\Omega\subset\mathbb{R}^2$ fitted with the boundary $\partial\Omega$. 
We denote the set of vertices of the mesh $\Th$ by $\Nh$, and the set of edges by $\Eh$. In particular $\Eh=\Eh^0\cup\Eh^D$, where $\Eh^D:=\{E\in\Eh:E\subset\Gamma^D\}$ is the set of Dirichlet boundary edges and $\Eh^0:=\Eh\setminus\Eh^D$. Similarly, we define $\Nh^D:=\{z\in\Nh:z\subset\Gamma^D\}$.
For any edge $E\in\Eh$, the midpoint of $E$ is denoted as $m_E$. Moreover, $\vnu:\Eh\to\mathbb{S}^1$ denotes the unit normal vectors to each $E\in\Eh$, where its orientation is arbitrary but remains fixed upon selection.
The Morley FEM was originally proposed in \cite{morley1968triangular} with the \emph{nonconforming} finite element space: 
\begin{align*}
\V_h=&\{v_h\in L^2(\Omega):v_h|_T\in\mathbb{P}_2(T), \forall T\in\Th, v_h\text{ is continuous at }z, \forall z\in\Nh \\&\nabla v_h\cdot\vnu\text{ is continuous at }m_E, \forall E\in\Eh\}.
\end{align*}  
Moreover, we define the discrete space of vanishing boundary values as:
\begin{equation}
\V_h^D=\{v_h\in\V_h: \text{dofs of } v_h \text{ are vanishing on }E\in\Eh^D\}.
\end{equation}
We next define the following useful quadrature rule for any $T\in\Th$:
\begin{equation}\label{eq:quad-rule}
Q_T(\phi):=\frac{|T|}{3}\sum_{E\in\Eh^T}\phi(m_E),
\end{equation}
where $\Eh^T$ denotes the set of three sides of triangle and $|T|$ denotes the area of $T$. $Q_T(\phi)$ is an approximation of $\int_T\phi$. 
We define the discrete counterparts $E_{pre}^h$ and $E_{bi}^h$ to $E_{pre}$ and $E_{bi}$ for discrete deformations $\vy_h\in[\V_h]^3$ as follows: 
\begin{equation}\label{eq:discrete-energy-prestrain}
E_{pre}^h\left[\vy_h\right]:= \sum_{T\in\Th}\frac{\mu}{12}\int_{T}\left|\g^{-\frac{1}{2}}D^2\vy_h\g^{-\frac{1}{2}}\right|^2+\frac{\lambda}{2\mu+\lambda}\tr\left(\g^{-\frac{1}{2}}D^2\vy_h\g^{-\frac{1}{2}}\right)^2,
\end{equation}
and 
\begin{equation}\label{eq:discrete-energy-bilayer}
E_{bi}^h\left[\vy_h\right]:= \sum_{T\in\Th}\frac{1}{2}\int_{T}\big|D^2\vy_h\big|^2-\sum_{i,j=1}^2\sum_{T\in\Th}Q_{T}\big(\partial_{ij}\vy_h\cdot(\partial_1\vy_h\times\partial_2\vy_h)Z_{ij}\big).
\end{equation}
In the practical implementation of all the proposed accelerated flows, considering the given $\vy^n_h\in[\V_h]^3$, the linearized constraint at $n$-th step is to restrict the increment in the tangent space
$\mathcal{F}_h(\vy_h^n):=\{\vv_h\in[\V_h^D]^3:L_T(\vy_h^n,\vv_h)=0, \forall T\in\Th\}$,
where
\begin{equation}\label{eq:linearized-cons}
L_T(\vy_h^n,\vv_h):=Q_T\big(\nabla \vv_h^T\nabla\vy_h^n+(\nabla\vy_h^n)^T\nabla\vv_h\big).
\end{equation}

We emphasis that all the algorithms presented above can be easily adapted to the fully discretized version by directly substituting everything with their discrete counterparts. The corresponding analysis in the spatially discretized cases is ignored but the general framework in this paper is applicable.

\section{Numerical results}
\label{sec:numerics}
In this section, we validate the properties of accelerated flows, examine the effect of key parameters $\alpha$ and $\beta$, as well as different choices of $\eta^n$ as adaptive and constancy in \eqref{eq:eta-nesterov} and \eqref{eq:eta-heavy-ball} respectively. We also compare our proposed algorithms with existing methods based on gradient flows.

\subsection{ {Benchmark examples}}\label{sec:numerical-ex}
We introduce the setup of several benchmark numerical examples by which we illustrate the performance of our new methods. 

\textbf{Example 1: Nonlinear Kirchhoff plates with vertical load}.
We take a square domain $\Omega=(0,4)^2$ and $g=I_2$ in the prestrained model, while adding a vertical load $\vf=(0,0,0.025)^T$. 
The corresponding energy is $E^h_{si}[\vy_h]:=E^h_{pre}[\vy_h]-\int_{\Omega}\vf\cdot\vy_h$, and it needs to be minimized over the admissible set \eqref{iso-admissible}, which includes a non-convex isometry constraint.
The plate is clamped on $\Gamma^D=\{0\}\times[0,4]\cup[0,4]\times\{0\}$, i.e. we prescribe the Dirichlet boundary condition with
\begin{equation}\label{eq:clamped-BC}
\vvarphi(x_1,x_2)=(x_1,x_2,0)^T \quad \mbox{and} \quad \Phi=[I_2,\mathbf{0}]^T \qquad (x_1,x_2) \in \Gamma^D.
\end{equation}
The initial state is the flat plate $\vy^0(x_1,x_2)=(x_1,x_2,0)^T$. This is the simplest test example for nonlinear plates as in previous works \cite{bartels2013approximation,bonito2019dg}. 

\textbf{Example 2: Bilayer plates with isotropic curvature}.
We consider a rectangular plate $\Omega=(-5,5)\times(-2,2)$ clamped on the side $\Gamma^D=\{-5\}\times[-2,2]$ with boundary condition \eqref{eq:clamped-BC}. The spontaneous curvature is isotropic and takes the form $Z=\gamma I_2$ with various values of $\gamma$. The exact solution to \eqref{bilayer-energy-3}, i.e. deformation with minimal energy, is a cylinder of radius $1/\gamma$ and energy $20\gamma^2$ \cite{schmidt2007minimal}. The initial state is the flat plate $\vy^0(x_1,x_2)=(x_1,x_2,0)^T$. In the rest of this section, we take $\gamma=1$ in simulations unless  otherwise specified (only in Section \ref{sec:local-global} we take $\gamma=5$). 

\textbf{Example 3: Prestrained plates with various ratios of eigenvalues}. We consider $\Omega=(-5,5)\times(-2,2)$ with the same boundary condition as in Example 2. We take metric $g$ to be 
\begin{equation}\label{eq:metric-ex3}
g(x_1,x_2)=\begin{bmatrix}
1+c^2(2(x_1+5)(x_1-2)+(x_1+5)^2)^2       & 0 \\
0       & 1
\end{bmatrix},
\end{equation}
with the parameter $c$ determining the ratios between two eigenvalues of $g(x_1,x_2)$ when $(x_1,x_2)$ is away from $\Gamma^D$. 
The initial state is $\vy^0(x_1,x_2)=(x_1,x_2,c(x_1+5)^2(x_1-2))$ satisfying $\I[\vy^0]=g$. We set Lam\'e parameters to be $\mu=12$ and $\lambda=0$. In the rest part of this section, we take $c=0.01$ or $0.1$, and the former corresponds to eigenvalues $3.56$ and $1$ on $x_1=5$,  while the latter gives eigenvalues $257$ and $1$ on $x_1=5$. With these choices, we will study the performance of algorithms when eigenvalues are comparable or have a relatively large ratio. 
   
\subsection{Properties of accelerated flows}\label{sec:numerics-properties}
We summarize the behavior of four cases under varying $\tau$ and a fixed spatial refinement of $512$ elements. The algorithms considered are: (i) Algorithm \ref{algo:acc_flow} with the choice $\eta^n = \frac{n-1}{n-1+\alpha}$, (ii) Algorithm \ref{algo:acc_flow} with the choice $\eta^n = 1 - \beta \tau$, (iii) Algorithm \ref{algo:acc_flow_backtrack}, as well as (iv) Algorithm \ref{algo:acc_flow_BDF2}. The results for these cases are presented in Tables \ref{tab:varying_tau_nesterov}, \ref{tab:varying_tau_heavyball}, \ref{tab:backtracking}, and \ref{tab:BDF2}, respectively.

We test these algorithms by Example 1 with $\tol=10^{-6}$, Example 2 with $\tol=10^{-4}$ and Example 3 with $\tol=10^{-6}$ and $c=0.01$ in \eqref{eq:metric-ex3}.      

In these simulations, we measure discrete counterparts $D^h_{g,p}$ of $D_{g,p}$, which is defined by 
\begin{equation}\label{eq:Dhgp}
D^h_{g,p}[\vy_h]=\Big(\sum_{T\in\Th}\Big|Q_T\big(\nabla\vy_h^T\nabla\vy_h-g\big)\Big|^p\Big)^{1/p}.
\end{equation}

We observe that $D^h_{g,p} \approx \mathcal{O}(\tau)$ in Tables \ref{tab:varying_tau_nesterov} and \ref{tab:varying_tau_heavyball}, consistent with Theorems \ref{thm:admissibility-pre}, \ref{thm:energy-stab-second-bi}, and Remark \ref{rmk:control-violation} for Algorithm \ref{algo:acc_flow} using both choices of $\eta^n$. 
Similarly, with the backtracking strategy, Algorithm \ref{algo:acc_flow_backtrack} also shows $D^h_{g,p} \approx \mathcal{O}(\tau)$ as seen in Table \ref{tab:backtracking}. For the BDF2 scheme, as in Table \ref{tab:BDF2}, $D^h_{g,p}$ is approximately $ \mathcal{O}(\tau^3) $ in Examples 1 and 2. 
However, in Example 3, both the computed energy and $D^h_{g,p}$ remain unchanged. This is because of the significant difference between $g$ and $I_2$, and a finer spatial discretization is required to observe a decrease in errors as $\tau$ decreases, as the spatial discretization error dominates in this case.

For a systematic discussion of constraint violations in high-order BDF schemes, we refer to our accompanying paper \cite{dong2024bdf}. Furthermore, $ N \approx \mathcal{O}(\tau^{-1}) $ for all cases with fixed tolerances, in agreement with Remark \ref{rmk:control-violation}.

Figure \ref{fig:Dh-plot} displays the change of $E_{bi}$ and $E_{total}$ over iterations in logarithmic scale for Example 2 computed by Algorithm \ref{algo:acc_flow}, confirming that the total energy decreases monotonically, while the potential energy exhibits oscillations. 

Furthermore, we investigate the behavior of Algorithm \ref{algo:acc_flow} for varying $\tol$ and a fixed $\tau$ in Table \ref{tab:kinetic-energy-tol} for Example 1. 
In particular, we observe that the discrete kinetic energy $E_{ki}^h[\vy^{N}_h]:=\frac{1}{2\tau^{2}}|\delta\vy^{N}_h|_{H^2}^2$ at the last step decreases  as $\tol$  becomes smaller, and is usually one magnitude smaller than $\tol$. This is consistent with Corollary \ref{cor:kinetic-lim}, where we show convergence of the kinetic energy to $0$ as $N$ tends to infinity. We also verify that $N$ increases without an upper bound as $\tol\to0$. Moreover, as $\tol\to 0$, $D^h_{I_2,1}$ appears to stabilize to a constant as in Table \ref{tab:kinetic-energy-tol}.

\begin{table}[h!]
\resizebox{\linewidth}{!}{
\begin{tabular}{|c|c|c|c|c|c|c|c|c|c|c|c|}
\hline
\multicolumn{4}{|c|}{Example 1}& \multicolumn{4}{|c|}{Example 2} &\multicolumn{4}{|c|}{Example 3} \\

\hline
$\tau$ & $E^h_{si}$ &  $D^h_{I_2,1}$ & $N$ & $\tau$ & $E^h_{bi}$ & $D^h_{I_2,2}$ & $N$ & $\tau$ & $E^h_{pre}$ & $D^h_{g,1}$ & $N$ \\
\hline
$2^{-3}$ & -1.01E-2 & 1.1E-3 & 85 &  $0.01$ & 17.2365 & 0.0753 & 15935 & 0.05 & 0.2124 & 0.1743 & 923  \\ 
\hline
$2^{-4}$ & -1.01E-2 & 6.1E-4 & 174 & $0.005$ & 17.3595 & 0.0374 & 32371 &  0.025 & 0.2112 & 0.0984 & 1848 \\ 
\hline
$2^{-5}$ & -1.01E-2 & 3.2E-4 & 348 & $0.0025$ & 17.4194 & 0.0186 & 66966  & 0.0125 & 0.2105 & 0.0578 & 3698 \\ 
\hline
\end{tabular}}
\caption{
Computation of  {Example 1,2,3 using Algorithm \ref{algo:acc_flow}} with $\eta^n=\frac{n-1}{n-1+\alpha}$ and varying $\tau$. We take $\alpha=3,6,3$ for Example 1,2,3 respectively.
} \label{tab:varying_tau_nesterov}
\end{table}

\begin{table}[h!]
\resizebox{\linewidth}{!}{
\begin{tabular}{|c|c|c|c|c|c|c|c|c|c|c|c|}
\hline
\multicolumn{4}{|c|}{Example 1}& \multicolumn{4}{|c|}{Example 2} &\multicolumn{4}{|c|}{Example 3} \\

\hline
$\tau$ & $E^h_{si}$ &  $D^h_{I_2,1}$ & $N$ & $\tau$ & $E^h_{bi}$ & $D^h_{I_2,2}$ & $N$ & $\tau$ & $E^h_{pre}$ & $D^h_{g,1}$ & $N$ \\
\hline
$2^{-3}$ & -1.01E-2 & 1.7E-3 & 56 &  $0.01$ & 17.0449 & 0.1211 & 9516 & 0.05 & 0.2122 & 0.2410 & 504  \\ 
\hline
$2^{-4}$ & -1.01E-2 & 9.4E-4 & 110 & $0.005$ & 17.2628 & 0.0601 & 19190 &  0.025 & 0.2110 & 0.1376 & 978 \\ 
\hline
$2^{-5}$ & -1.01E-2 & 4.9E-4 & 218 & $0.0025$ & 17.3704 & 0.0299 & 38543 & 0.0125 & 0.2104 & 0.0790 & 1925 \\ 
\hline
\end{tabular}}
\caption{
Computation of  {Example 1,2,3 using Algorithm \ref{algo:acc_flow}} with $\eta^n=1-\beta\tau$ and varying $\tau$. We take $\beta=1,0.2,0.8$ for Example 1,2,3 respectively.
} \label{tab:varying_tau_heavyball}
\end{table}

\begin{table}[h!]
\resizebox{\linewidth}{!}{
\begin{tabular}{|c|c|c|c|c|c|c|c|c|c|c|c|}
\hline
\multicolumn{4}{|c|}{Example 1}& \multicolumn{4}{|c|}{Example 2} &\multicolumn{4}{|c|}{Example 3} \\

\hline
$\tau$ & $E^h_{si}$ &  $D^h_{I_2,1}$ & $N$ & $\tau$ & $E^h_{bi}$ & $D^h_{I_2,2}$ & $N$ & $\tau$ & $E^h_{pre}$ & $D^h_{g,1}$ & $N$ \\
\hline
$2^{-3}$ & -1.01E-2 & 1.1E-3 & 33 &  $0.01$ & 17.1742 & 0.0921 & 6008 & 0.05 & 0.2124 & 0.1734 & 368  \\ 
\hline
$2^{-4}$ & -1.01E-2 & 5.7E-4 & 62 & $0.005$ & 17.3325 & 0.0454 & 12047 &  0.025 & 0.2111 & 0.0979 & 732  \\ 
\hline
$2^{-5}$ & -1.01E-2 & 2.9E-4 & 121 & $0.0025$ & 17.4097 & 0.0225 & 24142 & 0.0125 & 0.2104 & 0.0574 & 1458 \\ 
\hline
\end{tabular}}
\caption{
Computation of  {Example 1,2,3 using Algorithm \ref{algo:acc_flow_backtrack}} with $\alpha=3$, $512$ elements and varying $\tau$. 
} \label{tab:backtracking}
\end{table}

\begin{table}[h!]
\resizebox{\linewidth}{!}{
\begin{tabular}{|c|c|c|c|c|c|c|c|c|c|c|c|}
\hline
\multicolumn{4}{|c|}{Example 1}& \multicolumn{4}{|c|}{Example 2} &\multicolumn{4}{|c|}{Example 3} \\

\hline
$\tau$ & $E^h_{si}$ &  $D^h_{I_2,1}$ & $N$ & $\tau$ & $E^h_{bi}$ & $D^h_{I_2,2}$ & $N$ & $\tau$ & $E^h_{pre}$ & $D^h_{g,1}$ & $N$ \\
\hline
$2^{-3}$ & -1.01E-2 & 2.0E-5 & 34 &  $0.01$ & 17.4771 & 1.7E-5 & 7159 & 0.05 & 0.2098 & 0.0196 & 368  \\ 
\hline
$2^{-4}$ & -1.01E-2 & 1.9E-6 & 63 & $0.005$ & 17.4861 & 2.0E-6 & 12282 &  0.025 & 0.2098 & 0.0196 & 730 \\ 
\hline
$2^{-5}$ & -1.01E-2 & 1.9E-7 & 121 & $0.0025$ & 17.4859 & 2.4E-7 & 24372 & 0.0125 & 0.2098 & 0.0196 & 1458\\ 
\hline
\end{tabular}}
\caption{
Computation of  {Example 1,2,3 using Algorithm \ref{algo:acc_flow_BDF2}} with $\alpha=3$, $512$ elements and varying $\tau$. 
} \label{tab:BDF2}
\end{table}

\begin{figure}[htbp]
    \centering
    \begin{minipage}{0.45\textwidth}
        \centering
        \includegraphics[width=\textwidth]{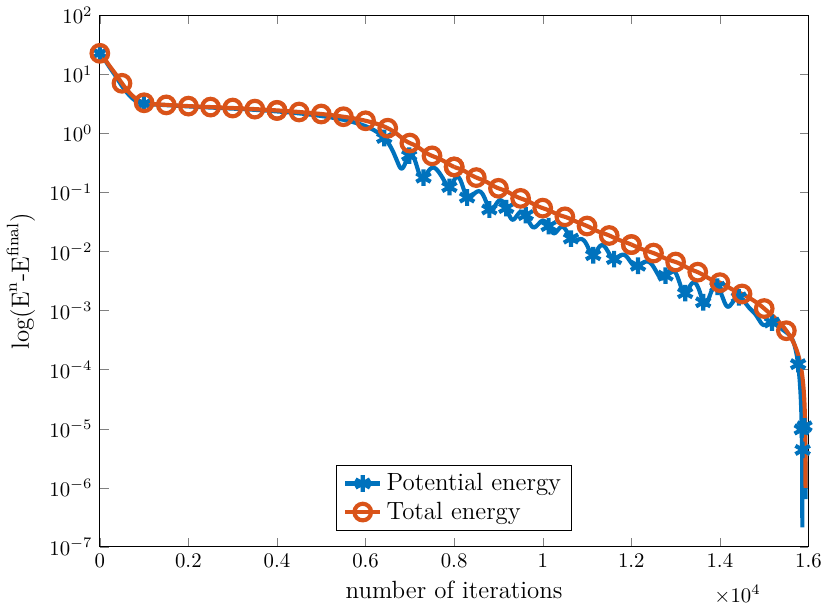} 
        \caption{Logarithmic plot of $E_{bi}$ and $E_{total}$ for Example 2 using Algorithm \ref{algo:acc_flow} with $\tau = 0.01$, $512$ elements, $\tol = 10^{-4}$, and $\alpha = 6$.}
\label{fig:Dh-plot}
    \end{minipage}\hfill
    \begin{minipage}{0.45\textwidth} 
        \centering
         \includegraphics[width=\textwidth]{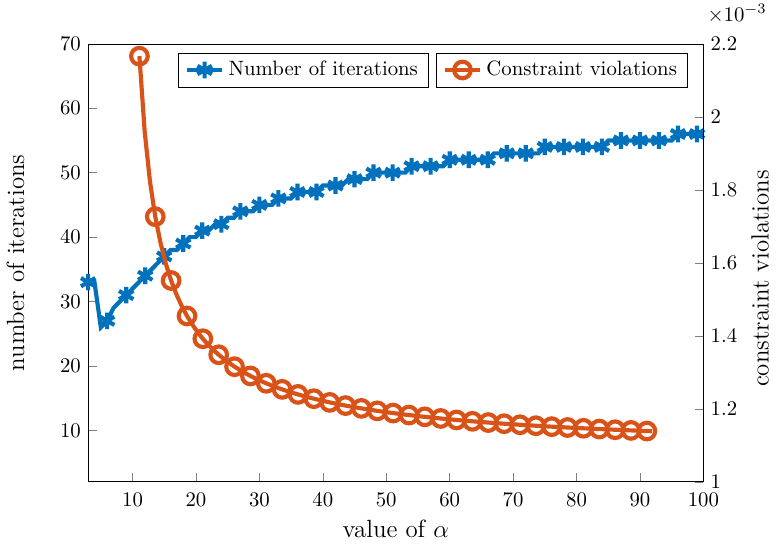} 
		\caption{Computation of Example 1 with $512$ elements and $\tau = 2^{-2}$ for various values of $\alpha$. Plots of $N$ versus $\alpha$ and $D^h_{I_2,1}$ versus $\alpha$.}
		\label{fig:alpha-N-Dh}
    \end{minipage}
\end{figure}

\subsection{Effect of $\alpha$ and $\beta$ and selection principle}\label{sec:varying-alpha}
With the choice $\eta^n = \frac{n-1}{n+\alpha-1}$ in Algorithm \ref{algo:acc_flow}, we demonstrate how $\alpha$ affects the total number of iterations and constraint violations in Example 1, as in Fig. \ref{fig:alpha-N-Dh}. Additionally, we present the effect of $\beta$ with the choice $\eta^n = 1 - \beta\tau$ in Example 2, as in Table \ref{tab:bilayer_nesterov_constant_eta}.

We note that the constraint violation decreases as $\alpha$ increases, and it stabilizes for large enough $\alpha$. Moreover, the constraint violation increases as $\beta$ decreases. These are consistent with estimates in the proof of Theorem \ref{thm:admissibility-pre}.

In numerical simulations, we select the value of $\alpha$ that results in the fastest convergence,  i.e., the fewest number of iterations $N$. It is important to note that the rationale behind designing Algorithm \ref{algo:acc_flow} is to accelerate the flow, and higher values of $\alpha$ or $\beta$ only lead to improvements of constants in constraint violation, not in its magnitudes.

We test $\alpha$ or $\beta$ for various examples, it seems that the optimal $\alpha$ or $\beta$ mainly depends on data of problems and are influenced slightly by $\tau$, but rather insensitive to mesh size $h$. 
Our parameter selection principle is as follows: for fixed $\tau$ and $\tol$, we use a relatively coarse mesh, treating the total number of iterations $N$ as a function of $\alpha$ or $\beta$. We employ the bisection method to find parameters that minimize $N(\alpha)$ or $N(\beta)$ within specified intervals, terminating the process after several steps to obtain a relatively optimal $\alpha$ or $\beta$.

\begin{table}[htbp]
    \centering
        \begin{minipage}{0.5\textwidth} % Adjust width as needed
        \centering
        \begin{tabular}{|c|c|c|c|}
\hline
$\tol$ & $E_{ki}^h[\vy^{N}_h]$ & $D^h_{I_2,1}$ & $N$ \\
\hline
1E-5 & 1.6E-6 & 1.13E-3 & 41 \\ 
\hline
1E-6 & 7.4E-7 & 1.14E-3 & 85 \\ 
\hline
1E-7 & 3.9E-8 & 1.14E-3 & 113 \\ 
\hline
1E-8 & 7.1E-9 & 1.14E-3 & 160 \\ 
\hline
\end{tabular}
        \caption{
        Kinetic energy at the final step, constraint violation, and iteration number of Algorithm \ref{algo:acc_flow} as a function of $\tol$ for Example 1, using $512$ elements, $\tau = 2^{-3}$, and $\alpha = 3$.
        } \label{tab:kinetic-energy-tol}
    \end{minipage}\hfill
    \begin{minipage}{0.45\textwidth} % Adjust width as needed
        \centering
        \begin{tabular}{|c|c|c|c|c|c|}
\hline
$\beta$ & $E_{bi}^h$ & $D^h_{I_2,2}$ & $N$  \\
\hline
$1$ &  17.2583  & 0.0653 & 34829 \\ 
\hline
$0.5$ &  17.1870 & 0.0833 & 18114 \\
\hline
$0.3$ & 17.1168 & 0.1014 & 11246 \\ 
\hline
$0.2$ & 17.0449 & 0.1211 & 9516 \\ 
\hline
$0.1$ & 16.9075 & 0.1627 & 10477 \\ 
\hline
\end{tabular}
\caption{Numerical performance of Algorithm \ref{algo:acc_flow} with $\eta^n = 1 - \beta \tau$ for varying $\beta$ in Example 2, using $512$ elements, $\tau = 0.01$, and $\tol = 10^{-4}$.
} \label{tab:bilayer_nesterov_constant_eta}
    \end{minipage}
\end{table}

\subsection{Acceleration effect: comparison with gradient flow}\label{sec:comparison}
We compare Algorithms \ref{algo:acc_flow}, \ref{algo:acc_flow_backtrack} and \ref{algo:acc_flow_BDF2} with gradient flow in the Table \ref{tab:comparison_method}, where relatively optimal $\alpha$ and $\beta$ are selected in the proposed algorithms following the principle in Section \ref{sec:varying-alpha}. Energy plots on a logarithmic scale for several cases are shown in Figure \ref{fig:comparison-vertical-load}.   

We observe that Algorithm \ref{algo:acc_flow} with varying damping coefficients $\eta^n = \frac{n-1}{n+\alpha-1}$ is slightly slower, 2-3 times faster, and slightly faster than the gradient flow for Examples 1, 2, and 3, respectively. The energy plot in Figure \ref{fig:comparison-vertical-load} shows that the gradient flow exhibits exponential decay in Example 1, which explains why Algorithm \ref{algo:acc_flow} is not faster in this case. In Example 2, the bilayer plates' energy introduces non-convexity, bringing difficulty to gradient flow computations, while accelerated flows show a significant advantage. In Example 3, the metric \eqref{eq:metric-ex3} has different eigenvalues compared to $g=I_2$ in Example 1, enabling the proposed algorithms to outperform the gradient flow. Figure \ref{fig:comparison-vertical-load} further illustrates the pronounced advantage of accelerated flows for metrics with a larger ratio of eigenvalues.

Algorithm \ref{algo:acc_flow} with constant coefficients $\eta^n = 1 - \beta\tau$ outperforms the algorithm with varying coefficients. However, the relatively optimal $\beta$ may yield a larger error in $D^h_{g,p}$, slightly reducing approximation accuracy.

In all scenarios, Algorithm \ref{algo:acc_flow_backtrack}, which incorporates backtracking, is faster than both Algorithm \ref{algo:acc_flow} and the one by gradient flow, particularly demonstrating significant speed improvements in Example 2 while maintaining comparable constraint violations.

When employing BDF2 schemes in Algorithm \ref{algo:acc_flow_BDF2}, the total number of iterations $N$ shows only a minor change compared to Algorithm \ref{algo:acc_flow_backtrack}, while significantly reducing constraint violations $D^h_{g,p}$, leading to more accurate approximations.

\begin{table}[h!]
\resizebox{\linewidth}{!}{
\begin{tabular}{|c|c|c|c|c|c|c|c|c|c|c|c|c|}
\hline
  & \multicolumn{4}{|c|}{Example 1} & \multicolumn{4}{|c|}{Example 2} & \multicolumn{4}{|c|}{Example 3} \\
\hline
Methods & param. & $E^h_{si}$ &  $D^h_{I_2,1}$ & $N$ & param. & $E^h_{bi}$ & $D^h_{I_2,2}$ & $N$ & param. & $E^h_{pre}$ & $D^h_{g,1}$ & $N$ \\
\hline
GF & N/A & -1.02E-2 & 2.1E-3 & 40 & N/A & 16.9116 & 0.1823 & 35458 & N/A & 0.2121 & 0.1860 & 727\\ 
\hline
Algo. \ref{algo:acc_flow} & $\alpha=6$ & -1.01E-2 & 8.0E-4 & 54 & $\alpha=9$ & 17.2749 & 0.0631 & 14832 & $\alpha=6$ & 0.2116 & 0.1188 & 542 \\ 
\hline
Algo. \ref{algo:acc_flow} & $\beta=1.375$ & -1.02E-2 & 1.4E-3 & 39 & $\beta=0.25$ & 17.0862 & 0.1096 & 10215 & $\beta=0.65$ & 0.2126 & 0.2938 & 364 \\ 
\hline
Algo. \ref{algo:acc_flow_backtrack} & $\alpha=3$ & -1.01E-2 & 1.1E-3 & 33 & $\alpha=3$ & 17.1742 & 0.0921 & 6008 & $\alpha=6.75$ & 0.2115 & 0.1121 & 334 \\ 
\hline
Algo. \ref{algo:acc_flow_BDF2} & $\alpha=3$ & -1.01E-2 & 2.0E-5 & 34 & $\alpha=3$ & 17.4771 & 1.7E-5 & 7159 & $\alpha=6.75$ & 0.2098 & 0.0196 & 332  \\ 
\hline
\end{tabular}}
\caption{  
Comparison of Algorithms \ref{algo:acc_flow}, \ref{algo:acc_flow_backtrack}, \ref{algo:acc_flow_BDF2}, and gradient flow (GF) with $\Omega$ partitioned into $512$ elements. For Examples 1, 2, and 3, we use $\tol = 10^{-6}, 10^{-4}, 10^{-6}$ and $\tau = 2^{-3}, 0.01, 0.05$, respectively. In Example 3, we set $c = 0.01$ in \eqref{eq:metric-ex3}.
}
\label{tab:comparison_method}
\end{table}

\begin{figure}[htbp]
	\begin{center}
	\includegraphics[width=.45\textwidth]{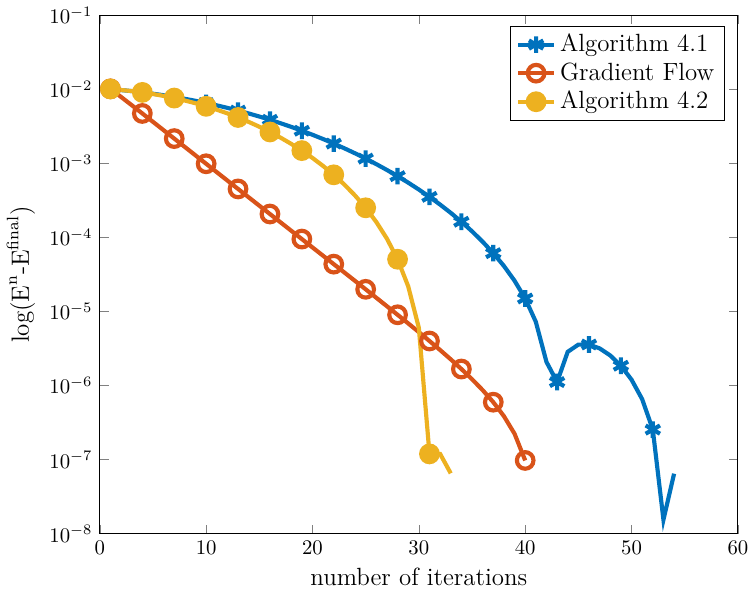}
    \includegraphics[width=.45\textwidth]{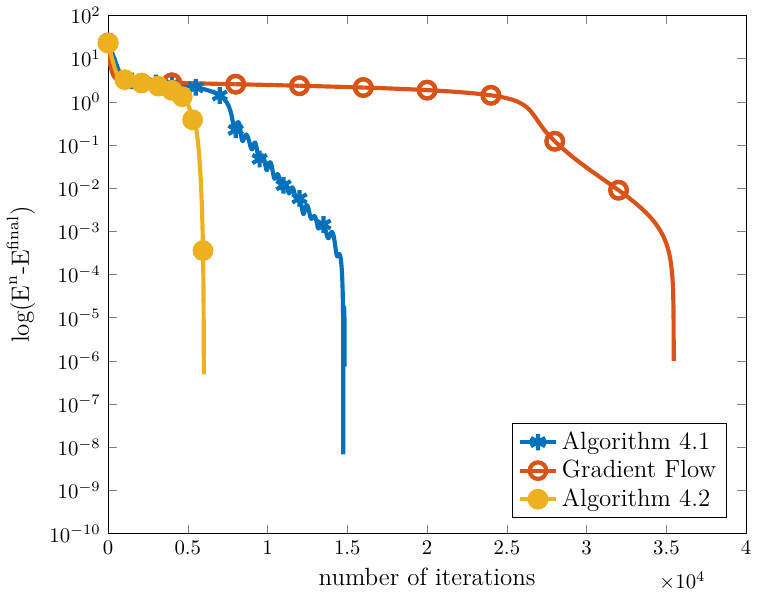}
    \includegraphics[width=.45\textwidth]{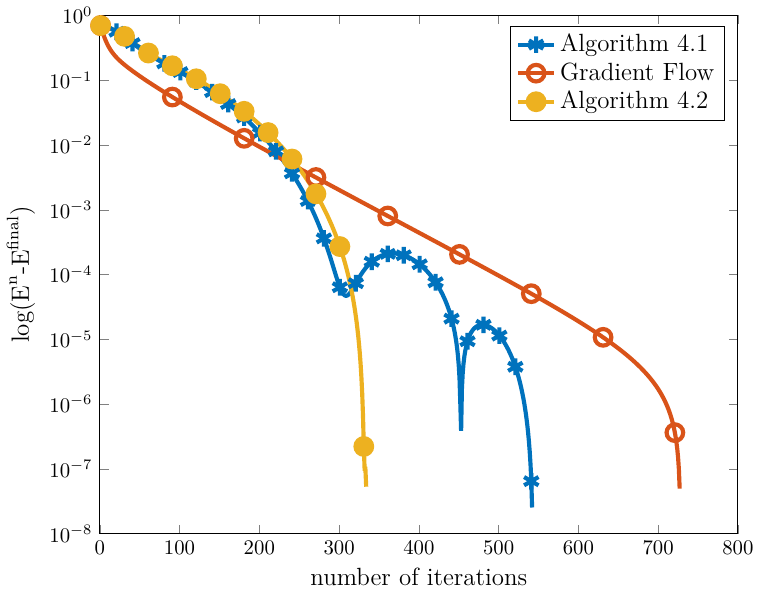}
    \includegraphics[width=.45\textwidth]{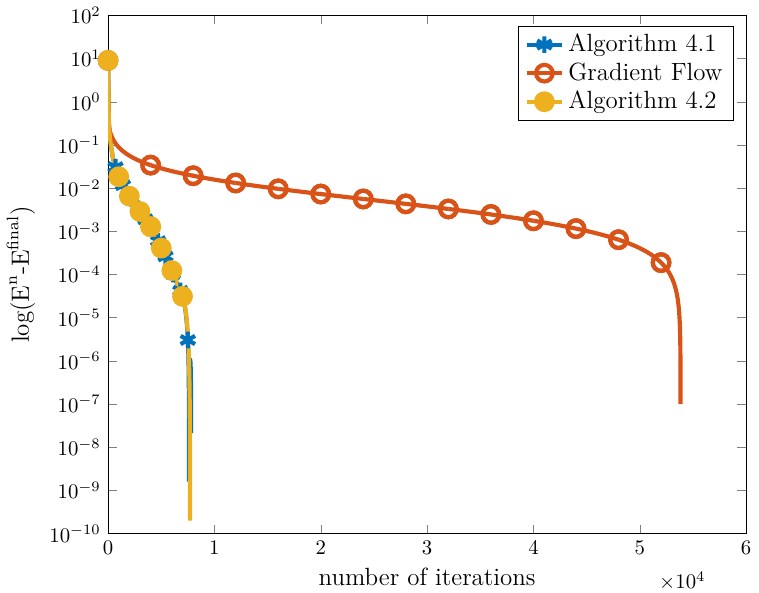}
		\caption{ 
        The potential energy plots in logarithmic scale for Algorithm \ref{algo:acc_flow} with varying damping coefficients, Algorithm \ref{algo:acc_flow_backtrack}, and gradient flows are displayed as follows: Example 1 in the upper left, Example 2 in the upper right, and Example 3 in the lower left, using the same data as in Table \ref{tab:comparison_method}. For Example 3, we further set $c=0.1$ in \eqref{eq:metric-ex3} to achieve a larger eigenvalue ratio, with the corresponding plot shown in the lower right.  {We use $512$ elements, $\tau=0.1$, and $\alpha=3$ for all the cases}.
  }
		\label{fig:comparison-vertical-load}
	\end{center}
\end{figure}

\subsection{Stationary points and global minimizers}\label{sec:local-global}
Computing global minimizers for these non-convex problems can be a challenging endeavor. When $Z=5I_2$ in the bilayer plates model, i.e. in Example 2, with $\gamma=5$, the exact solution is a cylinder with energy $E_{bi}=500$, but it winds more times compared to the case when $Z=I_2$. However, previous gradient flow-based methods, such as \cite{bartels2017bilayer,bonito2020discontinuous}, have been reported to fail in producing a cylindrical shape but stuck at a stationary point with a ``dog's ear'' shape. In contrast, the accelerated flows enable us to reach a proper approximation of the cylindrical shape. Snapshots of the evolution computed by Algorithm \ref{algo:acc_flow_BDF2} are displayed in Fig. \ref{fig:bilayer_Z5I}. This experiment was conducted with a refinement of $8192$ elements, $\tau=0.01$ and $\alpha=3$.  {We note that self-intersections occur during the evolution and in the final configuration. To prevent this, additional penalty terms must be incorporated into the energy functional; see \cite{bartels2022simulating} for detailed discussions.}
We plot the evolution of the potential energy $E_{bi}$ for this example in Figure \ref{fig:energy-local-global}. The energy curve exhibits a relatively flat region in the interval $N \in [4k, 12k]$, corresponding to a stationary point with a ``dog's ear'' shape, before decreasing to the cylindrical global minimizer. The energy $E_{bi}$ is monotonically decreasing due to the backtracking strategy in Algorithm \ref{algo:acc_flow_BDF2}. The final computed energy is $E^h_{bi} \approx 494.7594$, close to the exact energy of $500$, with a final constraint violation of $D^h_{I_2,2} = 1.9425 \times 10^{-4}$.

\begin{figure}[htbp]
	\begin{center}
	\includegraphics[width=.95\textwidth]{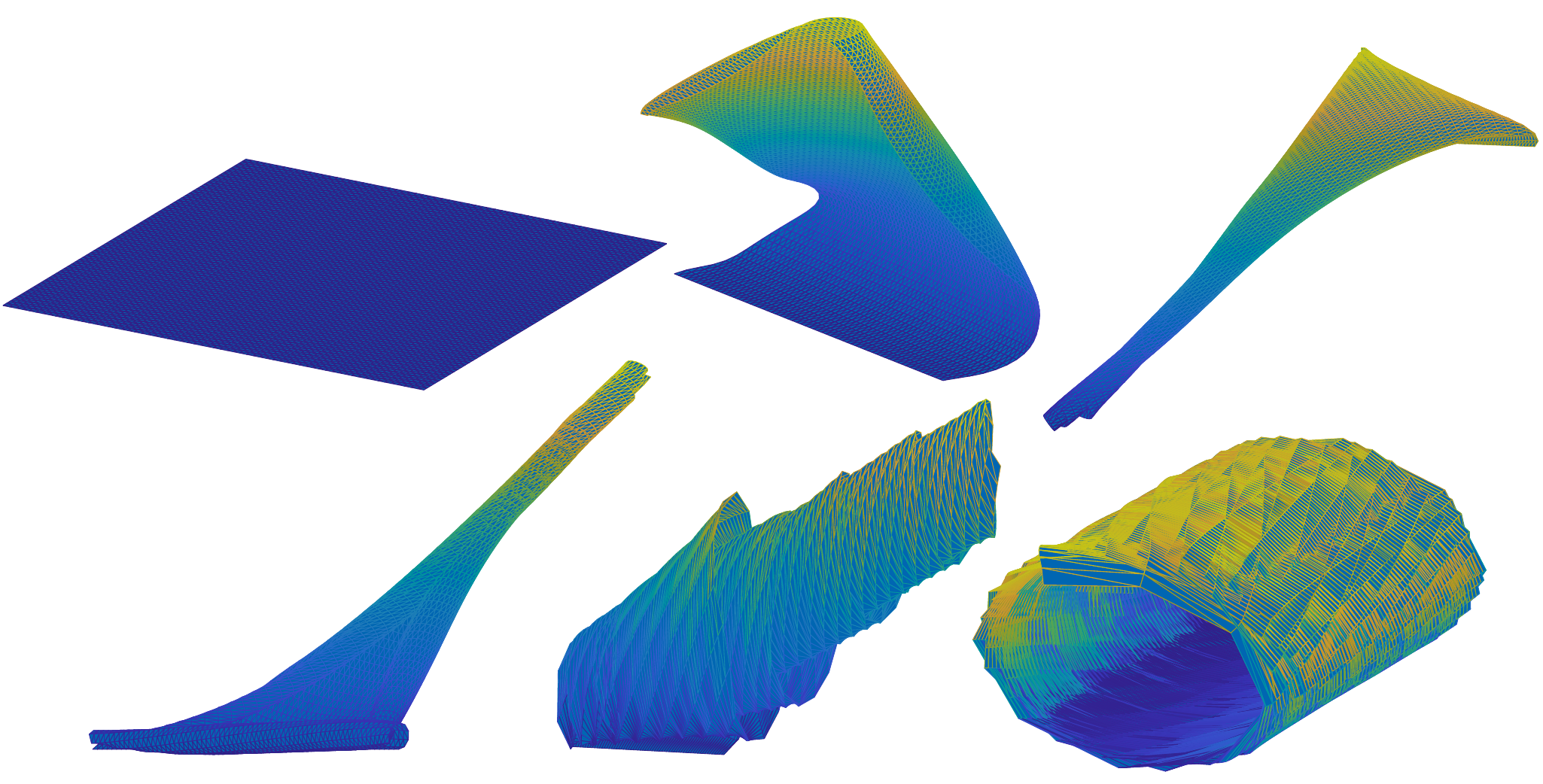}
		\caption{Snapshots of evolution for bilayer plates with $Z=5I_2$. From left to right, first row to second row: $\vy^n$ for $n=0,400,2500,6000,20000,30400$. The ``dog's ear'' shape appears in the process as a stationary point, but the accelerated flow is able to drive the deformation to evolve into the cylindrical global minimizer in the end.}
		\label{fig:bilayer_Z5I}
	\end{center}
\end{figure}

\begin{figure}[htbp]
	\begin{center}
	\includegraphics[width=.95\textwidth]{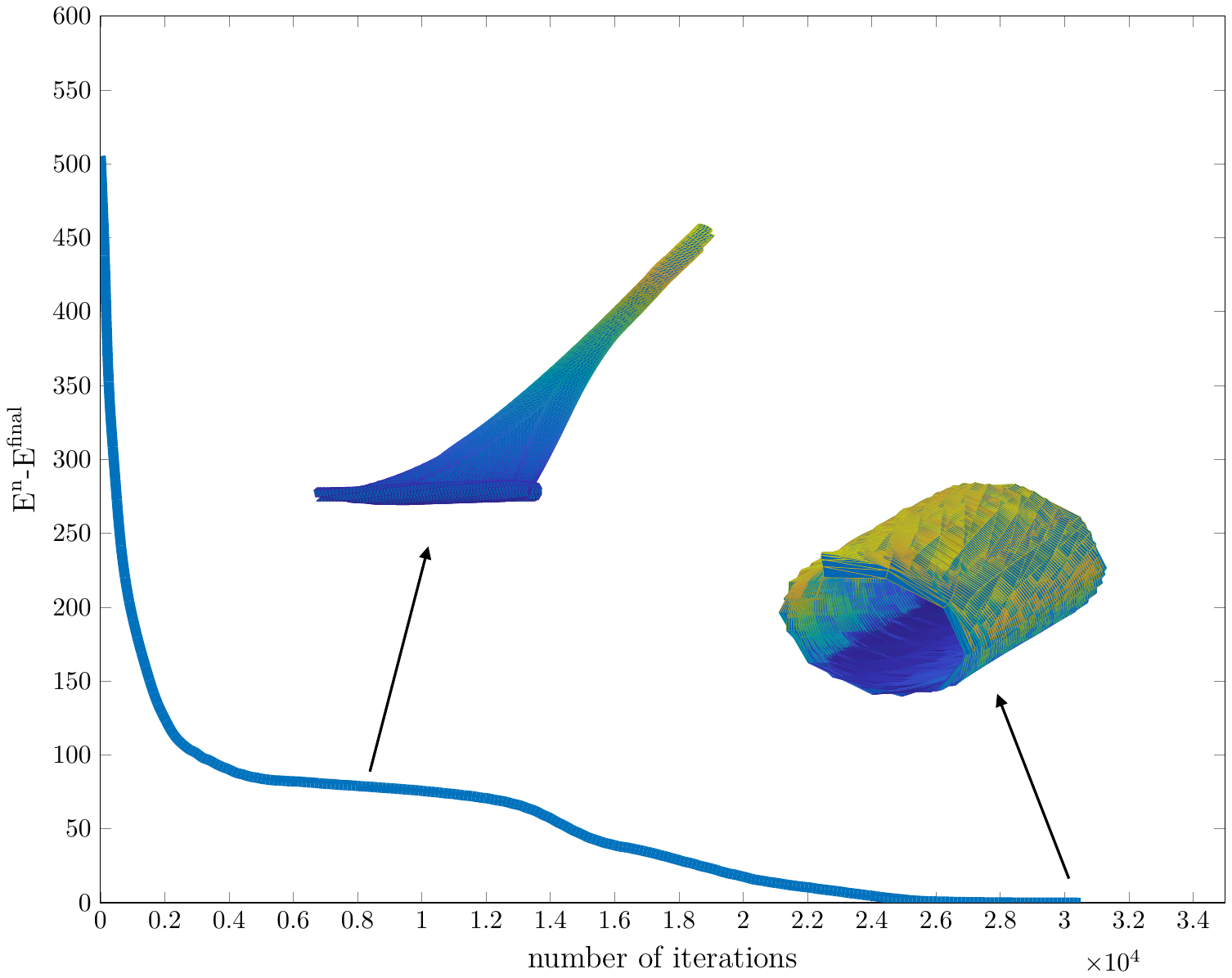}
		\caption{Plot of potential energy $E_{bi}$ for Example 2 with $\gamma = 5$, computed using Algorithm \ref{algo:acc_flow_BDF2}.}
		\label{fig:energy-local-global}
	\end{center}
\end{figure}

\section{Conclusions and future work}\label{sec:conclusions}
\subsection{Summary of contributions}
In this work, we have integrated multiple strategies ranged from optimization to numerical PDEs for approximating solutions of non-convex constrained variational problems. 
Our novel approaches resort to second-order dynamics motivated by acceleration techniques from convex optimization, together with a tangent space update strategy for effectively handling non-convex constraints, the backtracking technique, and the BDF2 approximation. 
By combining these techniques, we have developed a set of accelerated gradient flows that exhibit remarkable enhancements compared to existing methods based on gradient flows associated with nonlinear plates. 
 {Numerical experiments indicate that Algorithm~\ref{algo:acc_flow_BDF2} achieves better accuracy and efficiency than other proposed methods tested in the challenging bilayer case (Example~2). Specifically, the BDF2 scheme yields smaller constraint violations than BDF1, and the backtracking and restarting strategy reduces the number of iterations to reach the stopping criteria.}
  
\subsection{Extension to more general problems}
A natural question arises: can we extend the methods proposed in this work and the corresponding analysis to other non-convex variational problems with similar structures? We systematically investigated such methods in our accompanying paper \cite{dong2024bdf} for a general class of constrained variational problems with convex energies where the admissible sets consist of level sets of smooth operators; prestrained plate models \eqref{prob:min_Eg} serve as an application there. In particular, we establish a framework to estimate constraint violation errors for general high-order BDF schemes in \cite{dong2024bdf}, including their energy stability. However, extending these methods to non-convex energies is more challenging, as different types of non-convex/nonlinear terms in the energy may influence both the analysis and computational performance.

\subsection{Future directions}
While we have established energy stability and estimates on constraint violations in this article, several theoretical questions remain open. For instance the asymptotical convergence to stationary points of the constrained optimization problems as the time-step vanishes, the convergence rates of the proposed algorithms, and a full error analysis that incorporates both optimization and spatial discretization, will be of high interest to study.

A more delicate approach to enforce the metric constraint $I[\vy] = g$ might be desirable for exact preservation in iterative schemes while ensuring the stability, efficiency, and accuracy of the algorithms. That will bring us to a more challenging topic of studying accelerated gradient flows on smooth manifolds, which is on the agenda of our next steps.

\section*{Acknowledgments}
The authors thanks the anonymous reviewers whose comments helped to improve the presentation of the paper.
The work of G. Dong was supported by the NSFC grant  No. 12471402, and the NSF grant of Hunan Province No. 2024JJ5413. The work of H. G. was supported in part by the Andrew Sisson Fund and the Faculty Science Researcher Development Grant of the University of Melbourne. The work of S. Yang was supported by the NSFC grant No.12401512.

\appendix
\section{Proof of Theorem \ref{thm:energy-stab-second-bi}}\label{sec:appendix}
In this section, we present details of proof for the Theorem \ref{thm:energy-stab-second-bi} based on an induction argument. 
\begin{proof}
We proceed by an induction argument.
We note that it is straightforward to verify \eqref{eq:bilayer-second-energy-onestep} and \eqref{eq:cons-control-bi} for $M=1$ by a direct calculation. 
We assume \eqref{eq:bilayer-second-energy-onestep} and \eqref{eq:cons-control-bi} hold for $M=1,\cdots,n-1$ for $2\le n\le N-1$. Then our goal is to prove the same estimates for $M=n$. We first focus on the case when $\eta^n=(n-1)/(n+\alpha-1)$.  

{\bf Step (i)}, \emph{intermediate estimate}.
Taking $\vv=\delta\vy^{n+1}$ in \eqref{eq:bilayer_nesterov_scheme_3} and similar to \eqref{eq:energy-stab-pre-inter} we derive 
\begin{align}\label{eq:bilayer-energy-identity}
&\frac{1}{2\tau^{2}}\Big(|\delta\vy^{n+1}|_{H^2}^2-|\delta\vy^{n}|_{H^2}^2\Big)+\frac12a(\vy^{n+1},\vy^{n+1})-\frac12a(\vy^{n},\vy^{n}) \\  \nonumber
&+\frac{1-\eta^{n}}{2\tau^{2}}\Big(|\delta\vy^n|_{H^2}^2+|\delta\vy^{n+1}|_{H^2}^2\Big)+\frac{1+\eta^n}{2}a(\delta\vy^{n+1},\delta\vy^{n+1})+\frac{\eta^n}{2}a(\delta\vy^n,\delta\vy^{n})\\  \nonumber
&+\frac{\eta^{n}}{2\tau^{2}}|\delta\vy^{n+1}-\delta\vy^{n}|_{H^2}^2-\frac{\eta^n}{2}a(\delta\vy^{n+1}-\delta\vy^n,\delta\vy^{n+1}-\delta\vy^n)\\  \nonumber
&=\ell[\vy^n](\delta\vy^{n+1})+R^n(\delta\vy^{n+1}).
\end{align}  
Recalling $a(\vy,\vy)=|\vy|_{H^2}^2$, $0<\eta^n<1$, and taking $\tau<1$, it is clear that the second and third lines both are non-negative. Therefore, \eqref{eq:bilayer-energy-identity} becomes
{\small
\begin{equation}\label{eq:bi-rmv-23}
\frac{1}{2\tau^{2}}|\delta\vy^{n+1}|_{H^2}^2+\frac12|\vy^{n+1}|_{H^2}^2\le\frac{1}{2\tau^{2}}|\delta\vy^{n}|_{H^2}^2+\frac12|\vy^{n}|_{H^2}^2+\ell[\vy^n](\delta\vy^{n+1})+R^n(\delta\vy^{n+1}). 
\end{equation}
}
To control the right hand side, we use induction assumption and \eqref{eq:energy-stability-step1-bi} to yield
\begin{equation*}
 \frac{1}{2\tau^{2}}|\delta\vy^{n}|_{H^2}^2+E_{bi}[\vy^{n}]\le \frac{1}{2\tau^{2}}|\delta\vy^{1}|_{H^2}^2+E_{bi}[\vy^{1}]\le E_{bi}[\vy^{0}].
\end{equation*}
We note that $\vy^n\in\mathcal{A}^{\epsilon}_{I_2,2}$ with $\epsilon=\alpha_1\alpha_2 E_{bi}[\vy^0]$ as \eqref{eq:cons-control-bi} holds for $M=n$ and $\tau<\alpha_1$. Then resorting to Theorem \ref{lem:coercivity-bi} (coercivity) and requiring $\alpha_1\leq 1$, we derive 
\begin{equation}\label{eq:coer-est-2}
\frac{1}{2\tau^2}|\delta\vy^n|_{H^2}^2+\frac12|\vy^{n}|_{H^2}^2\le C(\vy^0,Z,\Omega)\alpha_2^2+C(\vy^0,\vvarphi,\Phi,Z,\Omega).
\end{equation}
Therefore, substituting \eqref{eq:coer-est-2} into the RHS of \eqref{eq:bi-rmv-23}, combining with estimates \eqref{eq:est-Rnh} and \eqref{eq:est-ln} for $R^n(\delta\vy^{n+1})$ and $\ell[\vy^n](\delta\vy^{n+1})$, and then using Cauchy inequality to produce $\frac12|\delta\vy^{n+1}|_{H^2}^2$ on the RHS of \eqref{eq:bi-rmv-23}, we obtain an intermediate estimate:
\begin{align}\label{eq:bilayer-interm-est}
&\qquad\qquad(\frac{1}{2\tau^{2}}-\frac12)|\delta\vy^{n+1}|_{H^2}^2+\frac12|\vy^{n+1}|_{H^2}^2 \\ \nonumber
&\le C(\vy^0,Z,\Omega)\alpha_2^4+C(\vy^0,\vvarphi,\Phi,Z,\Omega)\alpha_2^2+C(\vy^0,\vvarphi,\Phi,Z,\Omega),
\end{align}
and we emphasize that the constants $C(\vy^0,Z,\Omega),C(\vy^0,\vvarphi,\Phi,Z,\Omega)$ in \eqref{eq:coer-est-2} and \eqref{eq:bilayer-interm-est} are generic and not necessarily the same, but they depend on the same data specified within the parentheses.   

{\bf Step (ii)}, \emph{energy stability estimate}.
Using the algebraic relation 
\begin{equation*}
\begin{split}
(a^{M+1} -a^M)b^{M}&c^{M} +a^M(b^{M+1}-b^M)c^{M}+a^Mb^M(c^{M+1}-c^M)\\
&=a^{M+1}b^{M+1}c^{M+1}-a^{M}b^{M}c^{M}-(a^{M+1}-a^M)(b^{M+1}-b^M)c^{M+1}\\
&-(a^{M+1}-a^M)b^{M}(c^{M+1}-c^{M})-a^M(b^{M+1}-b^M)(c^{M+1}-c^M),
\end{split}
\end{equation*}
we rewrite $\ell[\vy^n](\delta \vy^{n+1})$ as follows:
{\small
\begin{equation}\label{eq:cubic-rewrite}
\begin{split}
 \ell[\vy^n](\delta \vy^{n+1})&=\sum_{i,j=1}^2\int_{\Omega}\big(\partial_{ij}\vy^{n+1}\cdot(\partial_1\vy^{n+1}\times\partial_2\vy^{n+1})-\partial_{ij}\vy^{n}\cdot(\partial_1\vy^{n}\times\partial_2\vy^{n})\big)Z_{ij} \\
& - \sum_{i,j=1}^2\int_{\Omega} \partial_{ij}\delta\vy^{n+1}\cdot(\partial_1\delta\vy^{n+1}\times\partial_2\vy^{n+1}+\partial_1\vy^{n}\times\partial_2\delta\vy^{n+1})Z_{ij} \\
& - \sum_{i,j=1}^2\int_{\Omega} \partial_{ij}\vy^{n}\cdot(\partial_1\delta\vy^{n+1}\times\partial_2\delta\vy^{n+1})Z_{ij}.
\end{split}
\end{equation} 
}
Substituting this into \eqref{eq:bilayer-energy-identity}, we note that the terms in first line are exactly the cubic energies $E^{nc}_{bi}[\vy^{n+1}]$, $E^{nc}_{bi}[\vy^{n}]$ and contribute to the
full energies $E_{bi}[\vy^{n+1}]$, $E_{bi}[\vy^n]$. In contrast, the remaining terms must be estimated and absorbed into terms involving $|\delta\vy^n|_{H^2}^2+|\delta\vy^{n+1}|_{H^2}^2$ in the second line of \eqref{eq:bilayer-energy-identity}. To this end,
we estimate and obtain 
\begin{equation*}
\label{eq:est-bi-almost-last}
\begin{split}
&\frac{1}{2\tau^{2}}\big(|\delta\vy^{n+1}|_{H^2}^2-|\delta\vy^{n}|_{H^2}^2\big)+E_{bi}[\vy^{n+1}]-E_{bi}[\vy^{n}]+\frac{1-\eta^{n}}{2\tau^{2}}\Big(|\delta\vy^{n+1}|_{H^2}^2+|\delta\vy^{n}|_{H^2}^2\Big)\\
&\quad \le C(\Omega,Z)|\delta\vy^{n+1}|_{H^2}^2\left(|\vy^{n+1}|_{H^2}+|\vy^{n}|_{H^2}+C(\vvarphi,\Phi)\right) \\
&\quad +C(\Omega,Z)\left(|\delta\vy^{n+1}|_{H^2}|\delta\vy^n|_{H^2}\left(|\vy^{n}|_{H^2}+C(\vvarphi,\Phi)\right)+ |\delta\vy^{n+1}|_{H^2}|\delta\vy^n|_{H^2}^2\right) \\ 
&\quad\le f_1(\alpha_2)|\delta\vy^{n+1}|_{H^2}^2+f_2(\alpha_2)|\delta\vy^{n}|_{H^2}^2.
\end{split}
\end{equation*}
For the first inequality we have used H\"older inequality, embedding of $H^1$ into $L^4$, and Poincar\'e inequality to estimate the second and third line of \eqref{eq:cubic-rewrite}, and used \eqref{eq:est-Rnh} to estimate $R^n(\delta\vy^{n+1})$.
For the second inequality, we use
intermediate estimate \eqref{eq:coer-est-2} and \eqref{eq:bilayer-interm-est}. The constants $f_1(\alpha_2),f_2(\alpha_2)$ are polynomials of $\alpha_2>0$ with coefficients $C(\vy^0,\vvarphi,\Phi,Z,\Omega)>0$ depending only on the problem data and independent of $n$, and thus they are monotonically increasing as $\alpha_2$ increases.

Our next goal is to find $\alpha_1$ such that when $\tau<\alpha_1$ the following estimate is valid:
\begin{equation}
\label{eq:falpha-est}
f_1(\alpha_2)|\delta\vy^{n+1}|_{H^2}^2+f_2(\alpha_2)|\delta\vy^{n}|_{H^2}^2\le\frac{1-\eta^{n}}{4\tau^{2}}\Big(|\delta\vy^{n+1}|_{H^2}^2+|\delta\vy^{n}|_{H^2}^2\Big).
\end{equation}
We note that $1-\eta^n\ge1-\frac{N-1}{N+\alpha-1}$. Exploiting the assumption $N=\alpha_0\tau^{-1}$, we derive $\frac{1-\eta^n}{4\tau^2}\ge\frac{\alpha}{4(\alpha_0\tau+(\alpha-1)\tau^2)}$.
As $\alpha_0>0$, $\alpha>1$, $f_1(\alpha_2),\; f_2(\alpha_2)>0$, we can conclude \eqref{eq:falpha-est} when $\max\{f_1(\alpha_2),f_2(\alpha_2)\}\leq \frac{\alpha}{4(\alpha_0\tau+(\alpha-1)\tau^2)}$ from which we infer $\tau<\frac{-\alpha_0+\sqrt{\alpha_0^2+\alpha(\alpha-1)/\max\{f_1(\alpha_2),f_2(\alpha_2)\}}}{2(\alpha-1)}$.
Recall that we have required $\alpha_1\leq 1$. This leads to the choice 
\[\alpha_1:=\min\left\{1, \frac{-\alpha_0+\sqrt{\alpha_0^2+\alpha(\alpha-1)/\max\{f_1(\alpha_2),f_2(\alpha_2)\}}}{2(\alpha-1)}\right\} ,\]
and $\alpha_1$ depends on $\vy^0,\vvarphi,\Phi,Z,\Omega,\alpha,\alpha_0$. 
Subsequently, \eqref{eq:falpha-est} further leads to the desired estimate \eqref{eq:bilayer-second-energy-onestep} for $M=n$.

{\bf Step (iii)}, \emph{control of constraint violation}. 
It remains to verify \eqref{eq:cons-control-bi} for $M=n$. 
Exploiting $\delta\vy^{M+1}\in\mathcal{F}(\vy^M)$, we derive 
\begin{equation*}
D_{I_2,2}[\vy^{M+1}]\le D_{I_2,2}[\vy^{M}]+\|\I[\delta\vy^{M+1}]\|_{L^2(\Omega)}. 
\end{equation*}
Summing over $M=0,\ldots,n$, using telescopic cancellation, considering H\"older inequality and embedding of $H^1$ into $L^4$, and recalling $\vy^0\in\A$, we obtain 
\begin{equation*}
D_{I_2,2}[\vy^{n+1}]\le C(\Omega)\sum_{k=0}^n|\delta\vy^{k+1}|_{H^2}^2.
\end{equation*}
Summing \eqref{eq:bilayer-second-energy-onestep} over $M=1,\ldots,n$, using \eqref{eq:energy-stability-step1-bi}, $\eta^M=\frac{M-1}{M+\alpha-1}$ and $n\le N$, we obtain 
\begin{equation*}
D_{I_2,2}[\vy^{n+1}]
\le C(\Omega,\alpha)N\tau^2E_{bi}[\vy^{0}].
\end{equation*} 
Recalling $N=\alpha_0\tau^{-1}$, this finishes the proof of \eqref{eq:cons-control-bi} for $M=n$ when $\alpha_2>C(\Omega,\alpha)\alpha_0$. Upon fixing $\alpha_2$, it provides a uniform choice (independent of $M$) of constants in the induction argument. This concludes the proof.  

{\bf Step (iv)}. When $\eta^n=1-\beta\tau$, we can readily obtain the stability \eqref{eq:bilayer-second-energy-onestep} and control of constraint violation \eqref{eq:cons-control-bi} by repeating step (i)-(iii) with simpler details and without the need for the assumption $N=\alpha_0\tau^{-1}$. We omit the details for brevity.   
\end{proof}

\bibliographystyle{siamplain}
\bibliography{morley_ref}
\end{document}